\documentclass[a4paper,reqno]{amsart}

\usepackage{amsmath,amsthm,amssymb, color, marginnote}
\usepackage[all]{xy}

\usepackage{graphicx,psfrag,epsfig}

\theoremstyle{plain}
\newtheorem{theorem}{Theorem}
\newtheorem*{maintheorem}{Main Theorem}
\newtheorem{lemma}[theorem]{Lemma}
\newtheorem{corollary}[theorem]{Corollary}
\newtheorem{proposition}[theorem]{Proposition}
\theoremstyle{definition}
\newtheorem{definition}[theorem]{Definition}
\newtheorem{remark}[theorem]{Remark}
\newtheorem*{assumption}{Assumptions}

\def\bdef{\begin{definition}}
\def\endef{\end{definition}}
\def\bthm{\begin{theorem}}
\def\ethm{\end{theorem}}
\def\blm{\begin{lemma}}
\def\elm{\end{lemma}}
\def\brm{\begin{remark}}
\def\erm{\end{remark}}
\def\bprop{\begin{proposition}}
\def\eprop{\end{proposition}}
\def\bcor{\begin{corollary}}
\def\ecor{\end{corollary}}
\def\be{\begin{eqnarray}}
\def\ee{\end{eqnarray}}
\def\beal{\begin{aligned}}
\def\enal{\end{aligned}}

\def\Om{\Omega}

\def\eps{\varepsilon}
\def\phi{\varphi}

\def\A{\mathbb A}
\def\R{\mathbb R}

\def\T{\mathbb T}
\def\Q{\mathbb Q}
\def\Z{\mathbb Z}

\def\L{\mathcal L}
\def\M{\mathcal M}

\def\~{\tilde}

\def\g{\gamma}

\def\dt{\delta}
\def\lb{\lambda}

\def\cE{\mathcal E}

\def\be{\begin{equation}}
\def\ee{\end{equation}}
\def\bdef{\begin{definition}}
\def\endef{\end{definition}}
\def\blm{\begin{lemma}}
\def\elm{\end{lemma}}
\def\beal{\begin{aligned}}
\def\enal{\end{aligned}}

\newtheorem*{Pf}{Proof}
\renewenvironment{proof}{\begin{Pf} \begin{upshape}} {\end{upshape} \qed\end{Pf}}

\title[On the M.L.S. of generic strictly convex billiard tables]{On {the} Marked length {spectrum} of generic strictly convex billiard tables}
\author{Guan Huang}
\author{Vadim Kaloshin} 
\author{Alfonso Sorrentino}
\address{Department of Mathematics, University of Maryland at College Park, College Park, MD 20740, US.}
\email{guan@math.umd.edu}
\address{Department of Mathematics, University of Maryland at College Park, College Park, MD 20740, US.}
\email{vadim.kaloshin@gmail.com}
\address{Dipartimento di Matematica, Universit\`a degli Studi di Roma ``Tor Vergata'', Rome, Italy.}
\email{sorrentino@mat.uniroma2.it}
\date{\today}
\begin{document}

\maketitle

\begin{abstract}In this paper we show that for a generic strictly convex domain, one can recover the eigendata corresponding to Aubry-Mather periodic orbits of the induced billiard map, from   the (maximal) marked length spectrum of the domain.
\end{abstract}
\section{Introduction}
A mathematical billiard is a system describing the inertial motion of a point mass inside a domain with elastic reflections at the boundary (which is assumed to have infinite mass). This simple model has been first proposed by Birkhoff \cite{Birkhoff} as a mathematical playground where: ``{\it the formal side, usually so formidable in dynamics, almost completely disappears and only the interesting qualitative questions need to be considered.}''

Since then billiards have become a very popular subject. Not only is their law of motion very physical and intuitive, but billiard-type dynamics is ubiquitous. Mathematically, they offer models in every subclass of dynamical systems (integrable, regular, chaotic, etc). More importantly, techniques initially devised for billiards have often been applied and adapted to other systems, becoming standard tools and having ripple effects beyond the field. 

 Moreover, despite their apparently simple (local) dynamics, 
their qualitative dynamical properties are extremely {non-local}! This global influence
on the dynamics translates into several intriguing {rigidity  phenomena}, which
are at the basis of many unanswered questions and conjectures. For instance, while the dependence of the dynamics on the geometry of the domain is well perceptible, an intriguing challenge is to understand to which extent dynamical information can be used to reconstruct the shape of the domain.
In this article, we will address this inverse problem in the case of periodic orbits in a strictly convex smooth planar domain $\Omega$.\\

The study of periodic orbits for billiard maps in strictly convex planar domains has been among the first dynamical features of billiards that have been investigated. One of the first results in the theory of billiards, for example, can be considered  Birkhoff's application of Poincare's last geometric theorem to show
the existence of infinitely many periodic orbits, which can be topologically distinguished in terms
of their {\it rotation number}\footnote{The rotation number  of a periodic billiard trajectory
is a rational number that can be roughly defined as
\[
\dfrac{p}{q}
\ =\
\dfrac{\text{winding number}}
{\text{number of reflections}}\ \in\ \big(0,\frac 12\Big],
\]
where the winding number $p>1$ is defined as follows.
Fix the positive orientation of $\partial \Om$ and
pick any reflection point of the closed geodesic on 
$\partial \Om$; then follow the trajectory and count
how many times it goes around $\partial \Om$ in
the positive direction until it comes back to
the starting point.
Notice that inverting the direction of motion for every
periodic billiard trajectory  of rotation number
$p/q  \in (0, 1/2]$, we obtain an orbit of rotation number
$(q-p)/q \in [1/2,1)$.}.
In \cite{Birkhoff} Birkhoff proved that for every rotation number $p/q \in (0, 1/2]$ in lowest terms,
there are at least two closed orbits of rotation number $p/q$: one maximizing the total length and the other obtained by min-max methods (see also  \cite[Theorem 1.2.4]{Siburg}).
This result is clearly optimal: in the case of a billiard in an ellipse, for example, there are only two periodic orbits of period $2$ (also called {\it diameters}), which correspond to the two semi-axis of the ellipse. However, it is easy to find cases in which there are more than two periodic orbits for a given rotation number: think, for example, of a billiard in a disk where, due to the existence of a $1$-dimensional group of symmetries (rotations), each periodic orbit generates a $1$-dimensional family of similar ones (all diameters are periodic orbits with period $2$).\\

A natural question is to understand which information on the geometry of the billiard domain, 
the set of periodic orbits does encode. More ambitiously, one could wonder whether a complete 
knowledge of this set  allows one to reconstruct the shape of the billiard  and hence the whole of the dynamics.\\

Let us start by introducing the {\it length spectrum} of a domain $\Om$.\\

\begin{definition}[{\bf Length Spectrum}] {\it
Given a domain $\Omega$, the length spectrum of $\Omega$ is given by the set  of lengths
of its periodic orbits, counted with multiplicity:
\[
\L_\Om:=\mathbb N\cdot \{\text{lengths of closed geodesics in }\Om\}
\cup \mathbb N \cdot \ell(\partial \Omega),
\]
where $\ell(\partial \Omega)$ denotes the length of the boundary.}\\
\end{definition}

\begin{remark}
A remarkable relation exists between the length spectrum
of a billiard in a convex domain $\Omega$  and the spectrum of the Laplace operator
in $\Omega$ with Dirichlet boundary condition (similarly for Neumann boundary one):
\begin{equation}\label{DirichletProblem}
\left\{
\begin{array}{l}
\Delta f = \lb f \quad \text{in}\; \Om \\
f|_{\partial \Om} = 0.\end{array}\right.\\
\end{equation}

From the physical point of view, the eigenvalues $\lb$
are the eigenfrequencies of the membrane $\Om$ with
a fixed boundary.

K. Andersson and R. Melrose \cite{AM}
proved the following relation between the Laplace spectrum 
and the length spectrum. Call the function 
\[
w(t):=\sum_{\lb_i \in spec \Delta}
cos (t \sqrt{-\lb_i}),
\]
the wave trace.
\\

\noindent {\bf Theorem (Andersson-Melrose).}
{\it The wave trace $w(t)$  is a well-defined 
generalized function (distribution) of $t$, smooth 
away from the length spectrum, namely, 
\begin{equation}\label{AndersonMelroseformula}
\mbox{\it sing. \!\!\!\! supp.} \big( w(t) \big)\subseteq  \pm \L_\Om \cup \{0\}.
\end{equation}
So if $l > 0$ belongs to the singular support of 
this distribution, then there exists either a closed billiard 
trajectory of length $l$, or a closed geodesic of length 
$l$ in the boundary of the billiard table.}\\
\end{remark}

Generically, {equality holds in \eqref{AndersonMelroseformula}}. More precisely, if 
no two distinct orbits have the same length and 
the Poincar\'e map of any periodic orbit is non-degenerate, 
then the singular support of the wave trace coincides 
with $\pm \L_\Om \cup \{0\}$ (see e.g. \cite{Popov}).\\

This theorem implies that,  at least for generic domains, one can recover the length spectrum from the Laplace one. This relation between periodic orbits and spectral properties of the domain, immediately recalls a more famous spectral problem (probably the most famous): {\it Can one hear the shape of a drum?}, as formulated in a very suggestive way by Mark Kac \cite{Kac} (although the problem had been already stated by Hermann Weyl). More precisely: is it possible to infer information about the shape of a drumhead ({\it i.e.}, a domain) from the sound it makes ({\it i.e.}, the list of basic harmonics/ eigenvalues of the Laplace operator with Dirichlet or Neumann boundary conditions)?
This question has not been completely solved yet: there are several negative answers (for instance by Milnor \cite{Milnor}
and Gordon, Webb, and Wolpert \cite{GordonWebbWolpert}), 
as well as some positive ones.  

Hezari--Zelditch, {going in the
  affirmative direction,} proved in~\cite{HZ} that, given an
ellipse $\cE$, any one-parameter $C^\infty$-deformation
$\Om_\eps$ which preserves the Laplace spectrum (with
respect to either Dirichlet or Neumann boundary conditions)
and the $\mathbb Z_2\times\mathbb Z_2$ symmetry group of 
the ellipse has to be \emph{flat} ({\it i.e.}, all derivatives have to
vanish for $\eps = 0$). {Popov--Topalov \cite{PT} recently 
extended these results (see also \cite{Zelditch}).} Further 
historical remarks on the inverse spectral problem can  be also 
found in~\cite{HZ}.
In~\cite{Sarnak}, P. Sarnak conjectures that the set of
  smooth convex domains isospectral to a given smooth convex
  domain is finite; {{for a} partial progress on this question,  see
  \cite{DKW}.}

\vspace{20 pt}

One of the difficulties in working with the length spectrum   is that all of these information come in a non-formatted way. For example, we lose track of the rotation number corresponding to each length.
A way to overcome this difficulty is to ``organize'' this set of information in a more systematic way, for instance by associating to each length the corresponding rotation number.  This new set  is called the 
{\it marked length spectrum} of $\Omega$ and denoted by $\M\L_\Om$.\\

One could also refine this set of information by considering not the lengths of all orbits, but selecting some of them.  More precisely, for each rotation number $p/q$ in lowest terms, one could consider the maximal
length among those having rotation number $p/q$. We call this map  $\M\L^{\rm max}_\Om: \Q\cap\big(0,\frac 12\Big] \longrightarrow \R_+ $ the {\it maximal marked length spectrum}: 
\begin{eqnarray*}
\M\L^{\rm max}_\Om({p}/{q}) =   \max \Big \{ \mbox{lengths of periodic orbits with rot. number}\; p/q \Big \}.
\end{eqnarray*}
This map is closely related to Mather's minimal average action (or $\beta$-function) and we will explain it in Section  \ref{sec3} (see also \cite{Siburg, Sor2015}).\\

\subsection{Main result}

In \cite[pp. 677-678]{GM} V. Guillemin and R. Melrose ask whether the Length Spectrum  and the eigenvalues of the linearizations of the (iterated) billiard map at periodic orbits, constitute a complete set of symplectic invariants for the system.

 Our main result shows that for generic domains, the eigendata corresponding to Aubry-Mather periodic orbits ({\it i.e.}, periodic orbits of maximal perimeter among those with the same rotation number ) can be actually recovered from the (Maximal) Marked Length Spectrum. More precisely:\\

\begin{maintheorem} 
For a generic strictly convex $C^{\tau+1}$-billiard table $\Omega$ ($\tau\geqslant2$), we have that for each $p/q\in\mathbb{Q}\cap(0,1/2]$ in lowest terms:
\begin{enumerate}
\item {The following  limit exists
$$\lim_{N\to+\infty} \left[ \M\L^{\rm max}_\Om\left(\frac{Np}{Nq-1}\right)-N\cdot\M\L^{\rm max}_\Om\left(\frac{p}{q}\right) \right] = -B_{p/q},$$
where $B_{p/q}$ denotes the minimum value of Peierls' Barrier function of rotation number $p/q$ (see Section \ref{sec5}).}
\item Moreover:
$$\lim_{N\to+\infty}\frac{1}{N}\log\Big|\M\L^{\rm max}_\Om(\frac{Np}{Nq-1})-N\cdot\M\L^{\rm max}_\Om(\frac{p}{q})+B_{p/q}\Big|=\log\lambda_{p/q},$$
where $\lambda_{p/q}$ is the eigenvalue of the linearization of the Poincare return map at the Aubry-Mather periodic orbit with rotation number $\frac{p}{q}$. 
\end{enumerate}
\end{maintheorem}

{See Theorem \ref{theorem1bis} in Section \ref{sec4} for a rephrasing of item (2) in the previous theorem 
in terms of Mather's $\beta$-function (which will be introduced in Section \ref{sec3}}).\\

The set of generic billiard tables is a (Baire) generic set, {\it i.e.}, a set that contains a countable intersection 
of open dense sets. See  {Section \ref{sec4} for a precise set of genericity assumptions.} 

{
\begin{remark}
Notice that for exact area-preserving twist maps, all of the above objects (Aubry-Mather periodic orbits, 
Peierls' barrier and Mather's $\beta$-function) are well defined and the argument in the proofs continue to be valid. Hence, our Main Theorem could be rephrased in terms of generic $C^{\tau+1}$ smooth exact area-preserving twist map, for $\tau\ge 2$. However, being our primary interest in this problem  motivated by spectral questions in billiard dynamics, 
we have opted to focus the presentation of our main results on this context.
\end{remark}
}

\begin{remark}
A natural question is the following: does the limit in item (2) always exist? If yes, does it determine to the eigenvalue 
$\lb_{p/q}$?\\
In \cite{XZ} the authors show that for a generic domain every hyperbolic periodic orbit
admits some homoclinic orbit. This raises the following question:
 Can one recover 
the eigenvalue of the linearization of the Poincare return map at any hyperbolic periodic orbit for 
a generic domain from the Marked Length Specturm?  \\
See Remark  \ref{rm:general-case2} for a more explicit connection between homoclinic orbits and our construction, 
and a description of the obstacles that one needs to overcome in order to extend our result to {a more general setting}.\\
\end{remark}

\begin{remark}
Quite interestingly, our main result could be applied to  identify for which irrational rotation number 
there exists or does not exist an invariant curve ({\it i.e.,} a caustic) with that rotation number.  In \cite{Greene}, 
J. Greene conjectured a criterion to test the existence of such curves (nowadays called ``Greene's residue criterion''), 
which was tested numerically in the case of the standard map. We recall here a version of this criterion as 
conjectured in \cite{MacKay}.\\

{\it Let $f$ be a symplectic twist map of the annulus  and  let $\rho \in \R$ be an irrational number. Consider a sequence 
of rational numbers $\frac{p_n}{q_n} \longrightarrow \rho$ as $n$ goes to $+\infty$ and for any minimizing periodic 
point $X_n$ of rotation number $\frac{p_n}{q_n}$ associates to it its {\it residue}, given by
$
r_n= \frac{1}{4}\left(  2- {\rm Tr}(Df^{q_n}(X_n))\right).
$
Then, the limit $\lim_{n\rightarrow +\infty} |r_n|^{1/q_n}=\mu(\rho)$ exists. Moreover, $\mu(\rho)\leq 1$ if and only if  
there exists an invariant curve with rotation number $\rho$.}\\

In \cite[Theorem 3]{AB}, M.-A. Arnaud and P. Berger proved a part of this criterion (the ``only if''). More specifically, 
they proved that if 
$$\limsup_{n\rightarrow +\infty} |r_n|^{1/q_n}>1,$$ then there is no homotopically non-trivial 
invariant curve with rotation number $\rho$.\\
Our result allows one to obtain a lower-bound for this limsup at all irrational rotation number and hence apply the above result to deduce the
non-existence of invariant curves.
\end{remark}

\vspace{20 pt}

\noindent {\bf Organization of the article.}
For the reader's convenience, in Sections \ref{sec2} and \ref{sec3} we provide some background material on billiard maps and Aubry-Mather theory, as well as their mutual relation.\\
 In Section \ref{sec4} we explain our genericity assumptions and in Section \ref{sec5} we prove how to approximate Peierls' barrier by means of elements in the length spectrum 
 {and prove assertion (1) in  Main Theorem}. \\
All of this, will be exploited in Section \ref{sec6} for the proof of the assertion (2) in  Main Theorem.
{More in details, a large part of this section we will consists in the proof of sort of normal  form for  
Peierls' barrier (see Theorem \ref{theorem1bis}), which  will satisfy some generic non-degeneracy condition 
(Lemma \ref{generic-lemma}). This will be enough to prove our main result.
For the reader's convenience, we outline here the main ideas involved in the construction of the normal form.
{
\begin{itemize}
\item Fix an Aubry-Mather periodic orbit of rotation number $p/q$.  
\item Choose a sequence of Aubry-Mather periodic orbits of rotation number {$pN/(Nq-1), N\ge 2$},  
approximating a homoclinic point of the $p/q$ Aubry-Mather periodic orbit. 
\item It turns out that a properly chosen linear combination of perimeters has a well-defined limit given by 
a properly chosen Peierls' barrier.
\item Moreover, the speed of convergence to the limit determines the eigenvalues of the $p/q$ Aubry-Mather 
periodic orbit.\\
\end{itemize}}} 

\vspace{10 pt}

\noindent {\bf Acknowledgement} The authors thank Ke Zhang for useful conversations. {The authors are also grateful to the referee for her/his valuable remarks and suggestions.}
V.K. has been partially support of the NSF grant 
DMS-1402164. A.S. has been partially supported by the PRIN-2012-74FYK7 grant ``{\it Variational and perturbative aspects of nonlinear differential problem}''.

\vspace{40 pt}


\section{The billiard map}\label{sec2}

In this section we would like to recall some properties of the billiard map. We refer to \cite{Siburg, Tabach} for a more comprehensive introduction to the study of billiards.\\

Let $\Omega$ be a strictly convex domain in $\R^2$ with $C^{\tau+1}$ boundary $\partial \Omega$,
with $\tau\geq 2$. The phase space $M$ of the billiard map consists of unit vectors
$(x,v)$ whose foot points $x$ are on $\partial \Omega$ and which have inward directions.
The billiard ball map $f:M \longrightarrow M$ takes $(x,v)$ to $(x',v')$, where $x'$
represents the point where the trajectory starting at $x$ with velocity $v$ hits the boundary
$\partial \Omega$ again, and $v'$ is the { reflected velocity}, according to
the standard reflection law: angle of incidence is equal to the angle of reflection (figure \ref{billiard}).

\begin{remark}
Observe that if $\Omega$ is not convex, then the billiard map is not continuous.
Moreover, as pointed out by Halpern \cite{Halpern}, if the boundary is not at
least $C^3$, then the flow might not be complete.
\end{remark}

Let us introduce coordinates on $M$.
We suppose that $\partial \Omega$ is parametrized  by  arclength $s$ and
let $\g:  [0, l] \longrightarrow \R^2$ denote such a parametrization,
where $l=l(\partial \Omega)$ denotes the length of $\partial \Omega$. Let $\phi$
be the angle between $v$ and the positive tangent to $\partial \Omega$ at $x$.
Hence, $ M$ can be identified with the annulus $\A = [0,l] \times (0,\pi)$
and the billiard map $f$ can be described as

\begin{eqnarray*}
f: [0,l] \times (0,\pi) &\longrightarrow& [0,l] \times (0,\pi)\\
(s,\phi) &\longmapsto & (s',\phi').
\end{eqnarray*}

In particular $f$ can be extended to $\bar{\A}=[0,l] \times [0,\pi]$ by fixing
$f(s,0)=f(s,\pi)= {\rm Id}$,
for all $s$. 

\begin{figure} [h!]
\begin{center}
\includegraphics[scale=0.3]{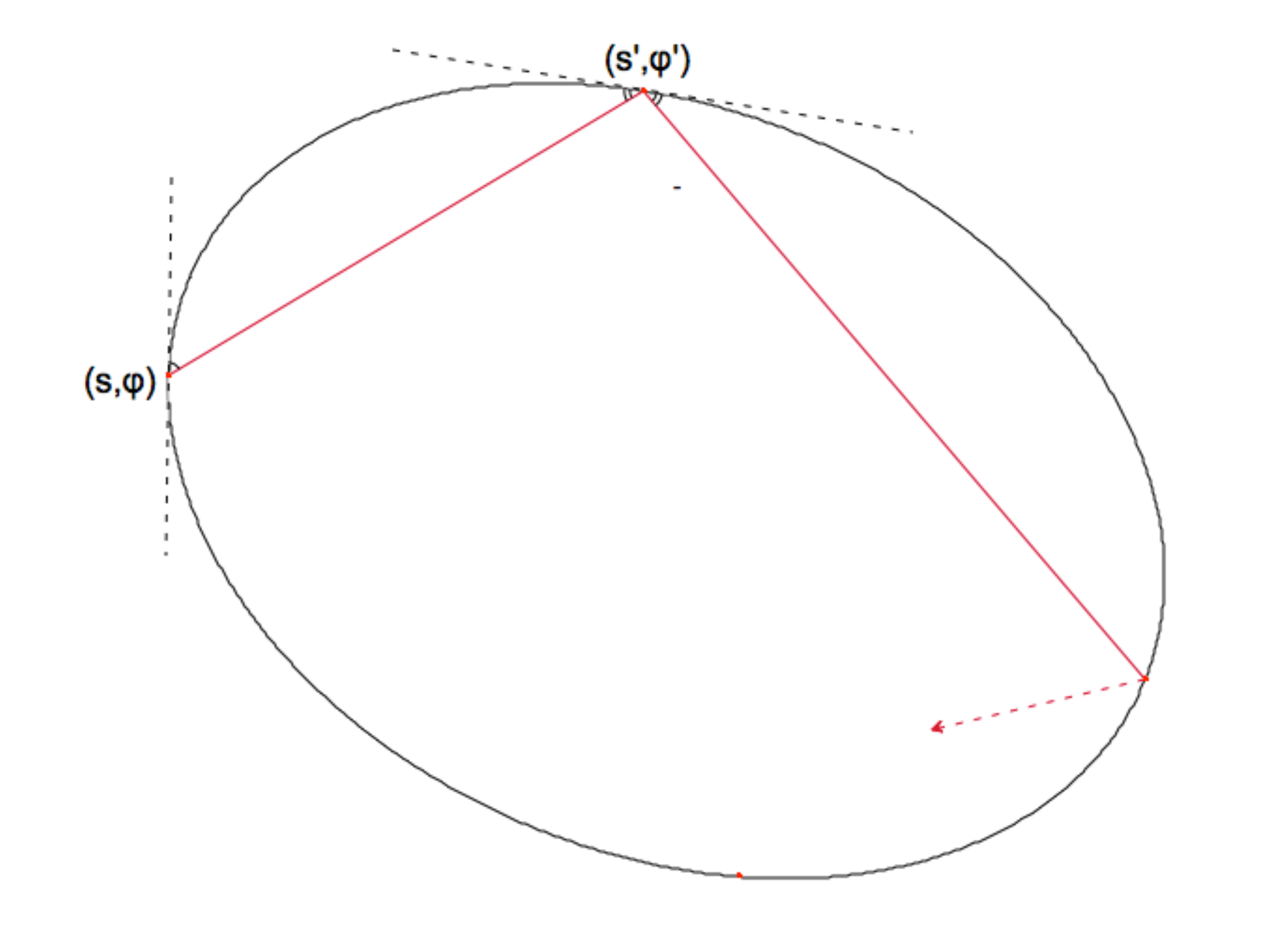}
\caption{}
\label{billiard}
\end{center}
\end{figure}

Let us denote by
\begin{equation}\label{length}
\ell(s,s') := \|\g(s) - \g(s')\|
\end{equation}
the Euclidean distance between two points on $\partial \Omega$. It is easy to prove that
\begin{equation}
\left\{ \begin{array}{l}
\dfrac{\partial \ell}{\partial s}(s,s') = - \cos \phi \\
\\
\dfrac{\partial \ell}{\partial s'}(s,s') = \cos \phi'\,.\\
\end{array}\right.
\end{equation}

\begin{remark}
If we lift everything to the universal cover and introduce new coordinates
$(\tilde{s},r)=(s, \cos \phi) \in \R \times (-1,1)$, then the billiard map is a twist map
with $\ell$ as generating function. See \cite{Siburg, Tabach}.\\
\end{remark}

Particularly interesting billiard orbits are {periodic orbits}, {\it i.e.}, billiard orbits
$X=\{x_k\}_{k\in\Z}:=\{(s_k,\phi_k)\}_{k\in \Z}$ for which there exists an integer $q\geq 2$ such that
$x_k=x_{k+q}$ for all $k\in \Z$. The minimal of such $q's$ represents the {\it period} of the orbit.
However periodic orbits with the same period may be of very different topological types. A useful topological invariant that allows one to distinguish amongst them is the so-called {\it rotation number}, which can be easily defined as follows. Let $X$ be a periodic orbit of period $q$ and consider the corresponding $q$-tuple $(s_1,\ldots, s_q) \in \R/l\Z$. For all $1\leq k \leq q$, there exists $\lambda_k \in (0,l)$ such that $s_{k+1}=s_k+\lb_k$  (using  the periodicity, $s_{q+1}=s_1$). Since the orbit is periodic, then $ \lb_1 +\ldots + \lb_k \in l\Z$ and takes value between $l$ and $(q-1)l$. The integer $p:= \frac{\lb_1 +\ldots + \lb_k}{l}$ is called the {\it winding number} of the orbit. The rotation number of $X$ will then be the rational number $\rho(X):=\frac{p}{q}$.
Observe that changing the orientation of the orbit replaces the rotation number $\frac{p}{q}$ by
$\frac{q-p}{q}$. Since, for the purpose of our result, we do not distinguish between two opposite orientations, then we can assume that $\rho(X) \in (0,\frac{1}{2}\big] \cap \Q$.\\

In \cite{Birkhoff}, as an application of Poincare's last geometric theorem, Birkhoff proved the following result.\\

\noindent{\bf Theorem [Birkhoff]}
{\it For every $p/q \in (0, 1/2]$ in lowest terms, there are at least two geometrically distinct periodic billiard trajectories with rotation number $p/q$}.\\

\begin{remark}
 In \cite{Lazutkin} V. Lazutkin  introduced a very special change of coordinates that reduces 
the billiard map $f$  to a very simple form.

Let $L_\Omega:  {[0,l]\times [0,\pi] \to  \T \times [0,\dt]}$  with small 
$\dt>0$ be given by
\begin{equation*}
L_\Omega(s,\phi)=\left(x=C^{-1}_\Omega \int_0^s \rho^{{2/3}}(s)ds,
\quad
y=4C_\Omega^{-1}\rho^{{-1/3}}(s)\ \sin \phi/2 \right),
\end{equation*}
where $\rho(s)$ is its radius of curvature at $s$ and $C_\Omega := \int_0^l \rho^{{2/3}} (s)ds$ is sometimes 
called the {\it Lazutkin perimeter} (observe that it is chosen so 
that period of $x$ is one).

In these new coordinates the billiard map becomes very simple (see \cite{Lazutkin}):

\begin{equation*}
f_L(x,y) = \Big( x+y +O(y^3),y + O(y^4) \Big).
\end{equation*}
In particular, near the boundary $\{\phi=0\} = \{y=0\}$, the billiard map $f_L$ reduces to
a small perturbation of the integrable map $(x,y)\longmapsto (x+y,y)$.

Using this result and a version of KAM theorem, Lazutkin proved in \cite{Lazutkin}
that if $\partial \Omega$
is sufficiently smooth (smoothness is determined by KAM theorem),
then there exists a positive measure set of invariant curves (corresponding to {\it caustics}), which accumulates on
the boundary and on which the motion is smoothly conjugate to a rigid rotation.\\
\end{remark}


\section{Aubry-Mather theory  and billiards.} \label{sec3}

At the beginning of the eighties Serge Aubry and John Mather developed, independently, what nowadays is commonly called {\it Aubry--Mather theory}.
This novel approach to the study of the dynamics of twist diffeomorphisms of the annulus, pointed out the existence of many {\it action-minimizing orbits}  for any given rotation number (for a more detailed introduction,  see  for example \cite{Bangert, MatherForni, Siburg, SorLecNotes}).

More precisely, let $f: \R/\Z \times \R \longrightarrow \R/\Z \times\R$ a monotone twist map, {\it i.e.}, a $C^1$ diffeomorphism such that its lift  to the universal cover $\tilde{f}$ satisfies the following properties (we denote $(x_1,y_1)= \tilde{f}(x_0,y_0)$):
\begin{itemize}
\item[(i)] $\tilde{f}(x_0+1, y_0) = \tilde{f}(x_0, y_0) + (1,0)$,
\item[(ii)] $\frac{\partial x_1}{\partial y_0} >0$  (monotone twist condition),
\item[(iii)] $\tilde{f}$ admits a (periodic) generating function $h$ ({\it i.e.},  it is an exact symplectic map):
$$
y_1\,dx_1 - y_0\,dx_0 = dh(x_0,x_1).
$$
\end{itemize}
In particular, it follows from (iii) that:

\begin{equation} \label{genfuncttwistmap}
\left\{
\begin{array}{l}
y_1 = \frac{\partial h}{\partial x_1}(x_0,x_1)\\
y_0 = - \frac{\partial h}{\partial x_0}(x_0,x_1)\,.
\end{array}
\right.
\end{equation}

\begin{remark}
The billiard map $f$ introduced above is an example of monotone twist map. In particular, its generating function 
is given by $h(x_0, x_1) = - \ell(x_0, x_1)$, where $\ell(x_0, x_1)$ denotes the euclidean distance between the two points on the boundary of the billiard domain corresponding
to $\gamma(x_0)$ and $\gamma(x_1)$.
\end{remark}

As it follows from (\ref{genfuncttwistmap}), orbits $x=\{x_i\}_{i\in\Z}$ of the monotone twist diffeomorphism $f$ correspond to critical points  of the {\it action functional}
$$
\{x_i\}_{i\in\Z} \longmapsto \sum_{i\in \Z} h(x_i, x_{i+1}).
$$

Aubry-Mather theory is concerned with the study of orbits that minimize this action-functional amongst all configurations with a prescribed rotation number; recall that the rotation number of an orbit $\{x_i\}_{i\in\Z}$ is given by $\pi \omega = \lim_{i\rightarrow \pm \infty} \frac{x_i}{i}$, if this limit exists (in the billiard case, this definition leads to the same  notion of rotation number introduced in subsection 1.2). In this context, {\it minimizing} is meant in the statistical mechanical sense, {\it i.e.}, every finite segment of the orbit minimizes the action functional with fixed end-points.\\

\noindent {\bf Theorem (S. Aubry, J. Mather).} {\it A monotone twist map possesses minimal orbits for every rotation number. For rational numbers there are always at least two periodic minimal orbits. Moreover, every minimal orbit lies on a Lipschitz graph over the $x$-axis.}\\

Let us denote by $\mathcal{M}_\omega$ the set of minimal trajectories $\underline{x}=\{x_i\}_{i\in\Z}$ with rotation number $\omega$ and by $\mathcal{M}^{\rm rec}_\omega$ the subset of recurrent ones. One can provide a detailed description of the structure of these sets (see \cite{Bangert, MatherForni}):
\begin{itemize}
\item  If $\omega \in \R\setminus \Q$, then $\mathcal{M}_\omega$ is totally ordered; moreover, there exist a map $f:\R \rightarrow \R$, which is the lift of an orientation-preserving circle homeomorphism  with rotation number $\omega$,  and a closed ${f}$-invariant set $A_\omega \subset \R$, such that $\mathcal{M}_\omega$ consists of the orbits of $f$ contained in $A_\omega$. Namely, $\underline{x}\in \mathcal{M}_\omega$ if and only if $x_0\in A_\omega$ and $x_i=f^i(x_0)$ for all $i\in \Z$. The projection $p_0$ (which to each $x=\{x_i\}_{i\in \Z}$ associates $x_0$) maps $\mathcal{M}_\omega$ homeomorphically into $A_\omega$.
Furthermore, $\underline{x}\in \mathcal{M}^{\rm rec}_\omega$ if and only if $x_0$ is a recurrent point of $f$. \\
\item If $\omega=\frac{p}{q}\in \Q$ (with $p$ and $q$ relatively prime), then $\mathcal{M}_\omega$ is the union of three disjoint and non-empty sets: 
$$\mathcal{M}^{\rm per}_\frac{p}{q} \cup
\mathcal{M}^{+}_\frac{p}{q} \cup
\mathcal{M}^{-}_\frac{p}{q}.$$
$\mathcal{M}^{\rm per}_\frac{p}{q}$ denotes the set of periodic minimal ones of rotation number $\frac{p}{q}$. We say that two elements 
$\underline{x}_-<\underline{x}_+$ of $\mathcal{M}^{\rm per}_\frac{p}{q}$ are neighboring if  there is no other element of $\mathcal{M}^{\rm per}_\frac{p}{q}$  between them.
We consider the sets $\mathcal{M}^{+}_\frac{p}{q}(\underline{x}_-,\underline{x}+)$ of all minimal orbits of rotation number $\frac{p}{q}$ that are asymptotic in the past ({\it i.e.}, 
as $i\rightarrow -\infty$) to $\underline{x}_-$ and in the future to $\underline{x}_+$. We define
$$\mathcal{M}^{+}_\frac{p}{q}= \bigcup_{(\underline{x}_-,\underline{x}_+)} \mathcal{M}^{+}_\frac{p}{q}(\underline{x}_-,\underline{x}_+),$$
where $(\underline{x}_-,\underline{x}_+)$ varies among all neighboring elements of $\mathcal{M}^{\rm per}_\frac{p}{q}$.\\
In a similar way, one defines $\mathcal{M}^{-}_\frac{p}{q}$ (just reverse the behaviours in the past and in the future).\\
Usually orbits in $\mathcal{M}^{\pm}_\frac{p}{q}$ are said to have rotation symbol $\frac{p}{q}\pm$.\\
\end{itemize}

We can now introduce the {\it minimal average action} (or {\it Mather's $\beta$-function}).

\begin{definition}\label{defbeta}
Let $x^{\omega} = \{x_i\}_{i\in\Z}$ be any minimal orbit with rotation number $\omega$. Then, the value of the {\em minimal average action} at $\omega$ is given by (this value is well-defined, since it does not depend on the chosen orbit):
\begin{equation}\label{avaction}
\beta(\omega) = \lim_{N\rightarrow +\infty} \frac{1}{2N} \sum_{i=-N}^{N-1} h(x_i,x_{i+1}).
\end{equation}
\end{definition}

This function $\beta: \R \longrightarrow \R$ enjoys many properties and encodes interesting information on the dynamics. In particular:
\begin{itemize}
\item[i)] $\beta$ is strictly convex and, hence, continuous (see \cite{MatherForni});
\item[ii)] $\beta$ is differentiable at all irrationals (see \cite{Mather90});
\item[iii)] $\beta$ is differentiable at a rational $p/q$ if and only if there exists an invariant circle consisting of periodic minimal orbits of rotation number $p/q$ (see \cite{Mather90}).\\
\end{itemize}

In particular, being $\beta$ a convex function, one can consider its convex conjugate:
$$
\alpha( c ) = \sup_{\omega\in \R} \left[ \omega \, c - \beta(\omega)\right].
$$

This function -- which is generally called {\it Mather's $\alpha$-function} -- also plays an important r\^ole in the study minimal orbits and in Mather's theory. We refer interested readers to surveys \cite{Bangert, MatherForni, Siburg, SorLecNotes}.\\

Observe that for each $\omega$ and $c$ one has:
$$
\alpha( c ) + \beta( \omega) \geq \omega  c,
$$
where equality is achieved if and only if $c\in \partial \beta(\omega)$ or, equivalently, if and only if  $\omega \in \partial \alpha (c )$  (the symbol $\partial$ denotes in this case the set of `subderivatives'  of the function, which is always non-empty and is a singleton if and only if the function is differentiable).\\

In the billiard case, since the generating function of the billiard map is the euclidean distance $-\ell$, the action of the orbit coincides -- up to a sign -- to the length of the trajectory that the ball traces on the  table $\Omega$. In particular, these two functions encode many dynamical  properties of the billiard ({see \cite{Siburg} for more details}):

\begin{itemize}
\item  For each $0<p/q\leq 1/2$, one has:
         \begin{equation}\label{betaandMLS}
         \beta(p/q) = - \frac{1}{q} {\mathcal ML}^{\rm max}_\Om({p}/{q}).
         \end{equation}
 \item $\beta$ is differentiable at $p/q$ if and only if there exists a caustic of rotation number $p/q$ ({\it i.e.}, all tangent orbits are periodic of rotation number $p/q$).
 \item If $\Gamma_{\omega}$ is a caustic with rotation number $\omega \in (0,1/2]$, then $\beta$ is differentiable at $\omega$ and $\beta'(\omega)= - {\rm length}(\Gamma_{\omega}) =: -|\Gamma_{\omega}|$ (see \cite[Theorem 3.2.10]{Siburg}). In particular,   $\beta$ is always differentiable at $0$ and $\beta'(0)= - |\partial \Omega|$.
 \item If $\Gamma_{\omega}$ is a caustic with rotation number $\omega \in (0,1/2]$, then one can associate to it another invariant, the so-called {\it Lazutkin invariant} $Q(\Gamma_{\omega})$.  More precisely 
\begin{equation}\label{lazinv}
Q(\Gamma_{\omega}) = |A-P| + |B-P| - |\stackrel \frown {AB} |
\end{equation}
 where $|\cdot |$ denotes the euclidean length and $|\stackrel \frown {AB}|$ the length of the arc on the caustic joining $A$ to $B$ (see figure \ref{lazutkin}).
 
This quantity is connected to the value of the $\alpha$-function. In fact, one can show that (see \cite[Theorem 3.2.10]{Siburg}):
 $$Q(\Gamma_\omega) = \alpha(  \beta'(\omega)  ) = \alpha( - |\Gamma_{\omega}|).$$
 
\begin{figure} [h!]
\begin{center}
\includegraphics[scale=0.7]{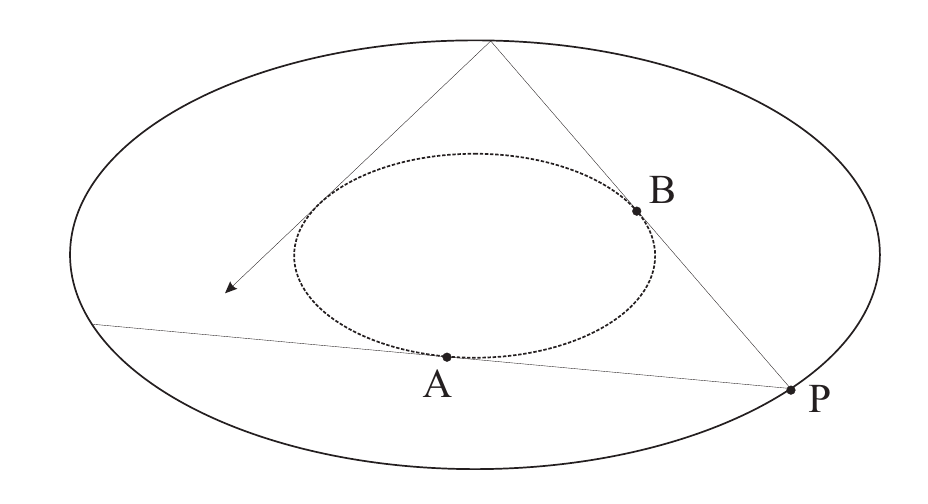}
\caption{Lazutkin invariant}
\label{lazutkin}
\end{center}
\end{figure}

 \end{itemize}


\section{The generic assumptions} \label{sec4}
Let $f:  (s,r)\to (s',r')$  denote the billiard map  corresponding to a strictly convex domain $\Omega$,  parametrized by arclength  $s$, and $h(s,s')=-\ell(s,s')$ (see \eqref{length}) denote the corresponding generating function. Then we have 
$$\begin{cases}r=-\partial_1h(s,s')\\ r'=\partial_2h(s,s').\end{cases}$$ 
Moreover
\begin{equation}\label{derivative1}Df(s,r)=\left(\begin{array}{cc}-\frac{\partial_{11}h(s,s')}{\partial_{12}h(s,s')}&-\frac{1}{\partial_{12}h(s,s')}\\
\partial_{12}h(s,s')-\partial_{22}h(s,s')\frac{\partial_{11}h(s,s')}{\partial_{12}h(s,s')}&\frac{-\partial_{22}h(s,s')}{\partial_{12}h(s,s')}\end{array}\right)\end{equation}
and
\begin{equation}\label{derivative2}Df^{-1}(s',r')=\left(\begin{array}{cc}-\frac{\partial_{22}h(s,s')}{\partial_{12}h(s,s')}&\frac{1}{\partial_{12}h(s,s')}\\
\partial_{11}h(s,s')\frac{\partial_{22}h(s,s')}{\partial_{12}h(s,s')}-\partial_{12}h(s,s')&\frac{-\partial_{11}h(s,s')}{\partial_{12}h(s,s')}\end{array}\right).\end{equation}
Here and after, we denote 
$$\partial_1h=\partial_{s}h,\quad \partial_{2}h=\partial_{s'}h,\quad \partial_{11}=\partial_{s}^2h,\quad\partial_{22}=\partial_{s'}^2h,\quad \partial_{12}h=\partial_s\partial_{s'}h.\\$$

\vspace{10 pt}

Let us describe our main generic assumptions:
\begin{assumption} 
For each $0<p/q\in\mathbb{Q}$ in lowest terms,
\begin{enumerate}
\item There exists a unique minimal periodic orbit in $\mathcal{M}^{\rm per}_\frac{p}{q}$.
\item The minimal periodic orbit is hyperbolic. 
\item The stable and unstable manifolds of the minimal periodic orbit intersect trasversally.
\end{enumerate}
\end{assumption}
Under these assumptions, we have the following well known  fact due to Aubry-Mather theory (see, e.g. \cite{MatherForni}).

\begin{proposition} \label{prop1}
For every $0<p/q\in\mathbb{Q}$ in lowest term,
there exists a unique minimal orbit in $\mathcal{M}^{+}_\frac{p}{q}$.\\
\end{proposition}

{Observe that in Proposition \ref{prop1}, the unique orbit in $\mathcal{M}^{+}_{{p}/{q}}$ connects the unique Aubry-Mather periodic orbit of rotation number $p/q$ to one of its shifts.\\
}

Let   $\tau\geqslant2$ and  denote $\mathcal{E}^{\tau}$ the set of all the strictly convex $C^{\tau+1}$-billiard tables, for which the corresponding billiard maps satisfy Assumptions 
{in Section \ref{sec4}}. {The set $\mathcal{E}^{\tau}$ is a residual subset  
of the space formed by strictly convex $C^{\tau+1}$-domains, with $C^{\tau+1}$-topology.  
See e.g. \cite{CKP2007}.}\\

Hereafter, we fix $\Omega\in\mathcal{E}^{\tau}$ and $f: (s,r)\to(s',r)$ is the associated billiard map.
Without further specification, all of our discussions are about the billiard map $f$.\\

\section{Approximation of the Barrier}\label{sec5}
In this section, we will prove statement (1) in  Main Theorem.

For  $\frac{p}{q}\in\mathbb{Q}\cap(0,\frac{1}{2}]$ in lowest term, 
let $$X_{p/q}:x_0,\dots,x_{q-1},$$ be the minimal periodic orbit with rotation number $\frac{p}{q}$ and let $L_{p,q}$ be its perimeter.  \\
Denote  by $L_{Np,Nq-1}$ be the perimeter of the minimal periodic orbit with rotation number $\frac{Np}{Nq-1}$. 
Then:

\begin{proposition}\label{prop2}
$$\lim_{N\to+\infty}L_{Np,Nq-1}-N\cdot L_{p,q}=-p/q\beta'_+(p/q)+\beta(p/q).$$ 
where $\beta(\cdot)$ is the minimal averaged action of the billiard map $f$ (introduced in  Definition \ref{defbeta}),  and $\beta'_+(\cdot)$ is its one-side derivative. \\
\end{proposition}

\begin{proof} {Recall relation \eqref{betaandMLS}}. Since $L_{p,q}=-q\beta(p/q)$ and $L_{Np,Nq-1}=-(Nq-1)\beta(\frac{Np}{Nq-1})$,
Then
\[\begin{split}L_{Np,Nq-1}-NL_{p,q}&=-[(Nq-1)\beta(\frac{Np}{Nq-1})-Nq\beta(p/q)]\\
&=-(Nq-1)(\beta(\frac{Np}{Nq-1})-\beta(p/q))+\beta(p/q)\\
&=-p/q\frac{\beta(\frac{Np}{Nq-1})-\beta(p/q)}{\frac{Np}{Nq-1}-\frac{p}{q}}+\beta(p/q)\\
&\longrightarrow \; -p/q\beta'_+(p/q)+\beta(p/q) \quad as\quad N\to+\infty.
\end{split}\]In the last equality, we have use the convexity  of the minimal averaged action $\beta(\cdot)$.
This proves the assertion of the lemma.
\end{proof}

\vspace{10 pt}

Let now $$X_{p/q+}: \dots,z_{-1},z_0,z_{1},\dots,$$ be the  minimal orbit in $\M^+_{\frac{p}{q}}$, and 
\begin{equation}\label{p/q+z0}d(f^{Nq}(z_0),f^{Nq}(x_1))\to 0,\quad d(f^{-Nq}(z_0),f^{-Nq}(x_0))\to 0,\quad as\quad N\to+\infty,\end{equation}
where $d(\cdot,\cdot)$ is the standard Euclidean distance in $\mathbb{R}^2$.

With slight abuse of notation, we will also use the same notation to denote the \mbox{$s$-coordinates} of the points in the orbits when they  are considered as  variables of the generating function $h(s,s')=-\ell(s,s')$. 
It follows from Aubry-Mather theory that  $X_{p/q+}$ minimizes 
\[\begin{split}B_{p/q}(z_0')&=\lim_{M,\;K\to+\infty}\sum _{i=-Kq+1}^{Mq-1}h\big(z_i',z_{i+1}'\big)-h(x_i,x_{i+1}) -h(x_{-Kq},x_{-Kq+1})\\
&=\lim_{M,\;K\to+\infty}\sum _{i=-Kq+1}^{Mq-1}h\big(z_i,z_{i+1}'\big)+(M+K)L_{p,q},\end{split}\]
 among all the configurations $\dots,z_{-1}',z_0',z_1',\dots$ such that (as $N\to+\infty$)
\begin{equation}\label{configuration}d(z_{-Nq+i}',x_i)\to0, \quad d(z_{Nq+i},x_{1+i})\to0,\quad i=0,\dots, q-1.  \end{equation}
\vspace{5 pt}

The function $B_{p/q}(\cdot)$ is usually referred as the {\it Peierls' Barrier function}.
Since the  periodic orbit  $X_{p/q}$ is hyperbolic, we have that  $B_{p/q}(z_0)$ is finite. \\

\begin{proposition}\label{prop:barrier}
$$\lim_{N\to+\infty}L_{Np,Nq-1}-NL_{p,q} = -B_{p/q}(z_0).\\$$
 \end{proposition}

{ \begin{remark}
 This result proves assertion (1) in Main Theorem.\\
 \end{remark}
 }

\begin{proof}For any $\epsilon>0$ and large enough $N\in\mathbb{N}$, $N/3<M<2N/3$, $K=N-M$, let $$X_{Np,Nq-1}:x'_{-Kq},\dots,x_0',\dots,x_{Mq-2}'$$ be the minimal periodic orbit with rotation number $\frac{Np}{Nq-1}$ and $d(x_{-Kq},x_0)<\epsilon$.
Then, clearly the configuration $$\dots x_{-2},x_{-1}X_{Np,Nq-1}x_1,x_2\dots$$satisfies \eqref{configuration}.
Therefore, by the minimality of the orbit $X_{p/q+}$, we have 
$$-(L_{Np,Nq-1}-NL_{p,q})\geqslant B_{p/q}(z_0)-C\epsilon,$$ 
where $C$ is a constant that depends only on the billiard map $f$.

On the other hand, the configuration $ z_{-Kq},\dots, z_0,\dots, z_{Mq-2},z_{-Kq}$ is of rotation number  $\frac{Np}{Nq-1}$; hence
$$-L_{Np,Nq-1}+NL_{p,q}\leqslant B_{p/q}(z_0)+C\epsilon.$$
Therefore, the assertion of the proposition follows.\\
\end{proof}

{Using Proposition \ref{prop2}, Proposition \ref{prop:barrier} and relation \eqref{betaandMLS}, observe that 
item (2) in Main Theorem} can be rephrased in terms of Mather's $\beta$-function in 
the following way. \\

\begin{theorem} \label{theorem1bis}
For a generic strictly convex $C^{\tau+1}$-billiard table $\Omega$ ($\tau\geqslant2$), we have that for each $p/q\in\mathbb{Q}\cap(0,1/2]$ in lowest terms:
$$\lim_{N\to+\infty}\frac{1}{N}\log\Big| (Nq-1) \beta (\frac{Np}{Nq-1})-Nq \beta(\frac{p}{q}){-}B_{p/q}\Big|=\log\lambda_{p/q},$$
where $\lambda_{p/q}$ is the eigenvalue of the linearization of the Poincare return map at the Aubry-Mather periodic orbit with rotation number $\frac{p}{q}$ and 
$B_{p/q}=p/q\beta'_+(p/q) - \beta(p/q)$. \\
\end{theorem}

\section{Eigenvalues of the Aubry-Mather periodic orbits}\label{sec6}
In this section, we continue to prove assertion  (2) of Main Theorem.\\

Let $\Lambda_{p/q}=Df^q(x_1).$
Since $X_{p/q}$ is hyperbolic, $\Lambda_{p/q}$ is hyperbolic, {\it i.e.}, it has two distinguished eigenvalues $0<\lambda_{\frac{p}{q}}<1$ and $\lambda_{\frac{p}{q}}^{-1}>1$.
{One of the main results of this section is the following theorem, which can be interpreted as a sort of normal form statement for Peierls' barrier.}

{\begin{theorem}\label{lemma-apr} There exists $N_{p,q}>0$, $C_{p,q}\in\mathbb{R}$ and $C_{p,q}'\in\mathbb{R}$ such that, if $N>N_{p,q}$,  there exists a periodic orbit $X_{Np,Nq-1}$ with minimal period $Nq-1$,  rotation number $Np/(Nq-1)$ and  perimeter $L_{Np, Nq-1}'$ satisfying, 
$$L_{Np,Nq-1}'-N\cdot L_{p,q}=-B_{p/q}(z_0)+C_{p,q}\lambda_{\frac{p}{q}}^N+\mathcal{O}(\lambda_{\frac{p}{q}}^{9N/8}),\quad \text{if $N$ is even},$$
and 
$$L_{Np,Nq-1}'-N\cdot L_{p,q}=-B_{p/q}(z_0)+C_{p,q}'\lambda_{\frac{p}{q}}^N+\mathcal{O}(\lambda_{\frac{p}{q}}^{9N/8}),\quad \text{if $N$ is odd}.$$
Moreover $d(z_0,X_{Np,Nq-1})=\mathcal{O}(\lambda_{\frac{p}{q}}^N)$.\\
\end{theorem}}
{
\brm 
Notice that the ``even'' $C_{p,q}$ can be different from the ``odd'' $C'_{p,q}$
(see \eqref{evenC} and \eqref{oddC} respectively).
\erm}

{We will  show in Lemma \ref{generic-lemma} that,  for a generic billiard table, the above constants $C_{p,q}$ and $C'_{p,q}$ are non-zero: this non-degeneracy property and Theorem \ref{lemma-apr} easily imply the proof of  assertion  (2) in Main Theorem (see  the end of this section).}\\

{In order to prove Theorem \ref{lemma-apr}, let us start by recalling} the following lemma, which is well known, see e.g \cite{stowe1986,zhang2011}.\\

\begin{lemma}\label{normal-form} For any $\epsilon>0$, there exists a $C^{1,\frac{1}{2}}$ diffeomorphism $\Phi: V\to U$, where $U,V$ are neighborhood of $x_1$ such that 
$$\Phi^{-1}\circ f^q\circ \Phi=\Lambda_{p/q}, \quad \|\Phi-{\rm Id}\|_{C^1}\leqslant \epsilon,\quad \text{and}\quad \|\Phi^{-1}-{\rm Id}\|_{C^1}\leqslant \epsilon.$$
Moreover,
$$\Phi(z)-\Phi(z')=z-z'+\mathcal{O}(\max\{|z|^{1/2},|z'|^{1/2}\}|z-z'|).\\$$
\end{lemma}

\vspace{20 pt}

{Let us start now the proof of Theorem \ref{lemma-apr}.}\\

\begin{proof}[{\bf Theorem \ref{lemma-apr}}]
From \eqref{p/q+z0}, we have that there exist $n_0$ and $m_0$ such that $f^{m_0q}(z_0)\in U$ and $f^{-n_0q+1}(z_0)\in U$. Let us denote  their images under $\Phi$ as $$A=\Phi(f^{m_0q}(z_0)) \quad \text{and} \quad B=\Phi(f^{-n_0q+1}(z_0)).$$ 
\begin{figure}[htb]\def\svgwidth{350pt}\label{saddle1}
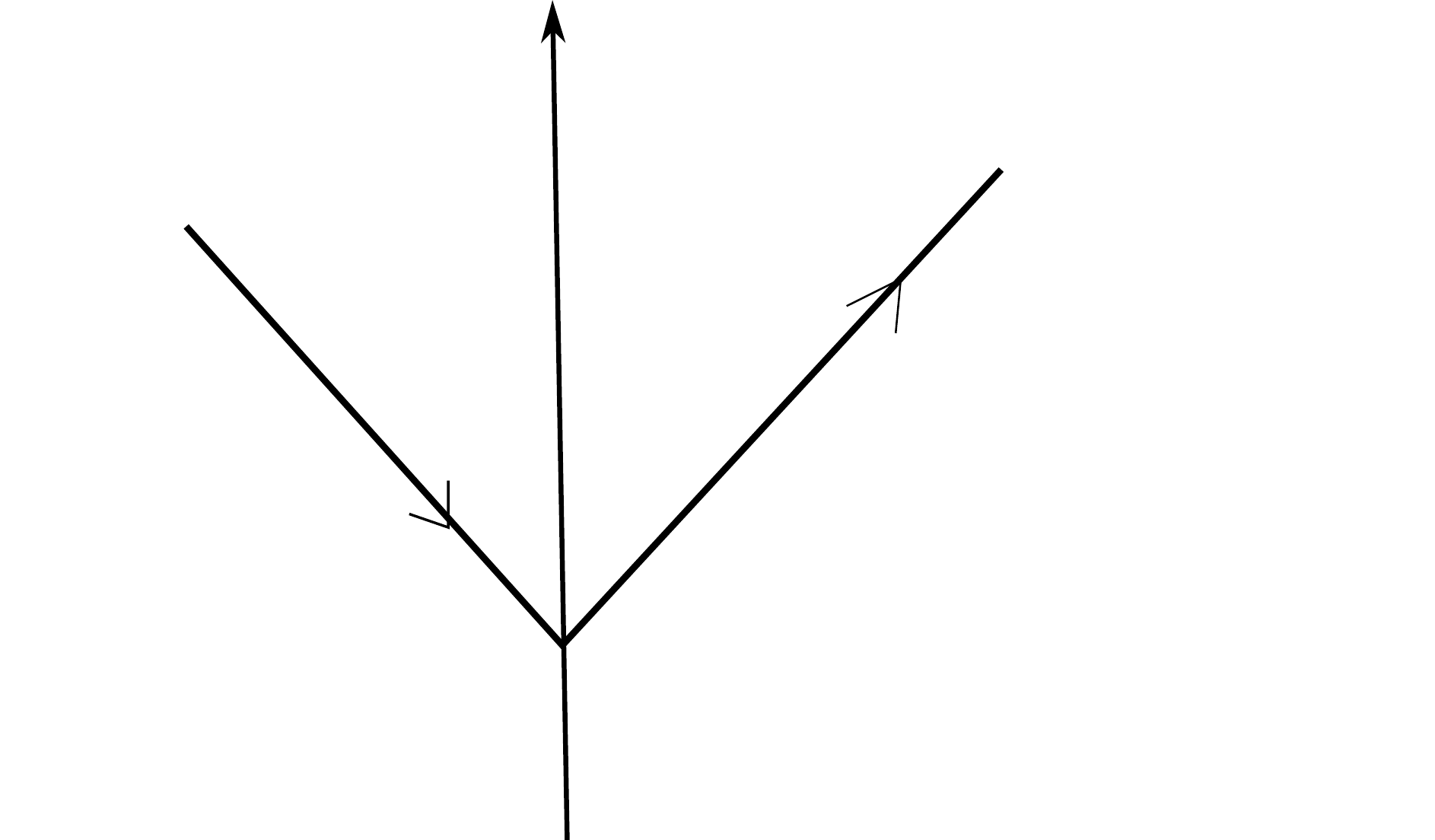
\caption{Saddle}
\label{saddle1}
\end{figure}

For the sake of simplicity, hereafter in this proof, we will write $\Lambda_{p/q}$ and $\lambda_{\frac{p}{q}}$  as $\Lambda$ and $\lambda.$ \\

Now we consider the standard $\bar{x}\text{-}\bar{y}$ plane, where $x_1$ is located at the origin $O$.  
The unit eigenvectors corresponding to the eigenvalues $\lambda$ and $\lambda^{-1}$  are respectively, 
$$\left(\begin{array}{c}-\sin\theta\\\cos\theta\end{array}\right)\quad \text{and}\quad\left(\begin{array}{c}\cos\theta\\\sin\theta\end{array}\right).$$  See  Figure~ \ref{saddle1}.
Using  the change of coordinates
$$R_{\theta}:\mathbb{R}^2\to\mathbb{R}^2, \quad \left(\begin{array}{c}\bar{\xi} \\ \bar{\eta}\end{array}\right)=R_{\theta}\left(\begin{array}{c}\bar{x}\\\bar{y}\end{array}\right):=\left(\begin{array}{cc}\cos\theta &\sin\theta \\ -\sin\theta&\cos\theta\end{array}\right)\left(\begin{array}{c}\bar{x} \\ \bar{y}\end{array}\right),$$
we tranform   the map $$\left(\begin{array}{c}\bar{x} \\ \bar{y}\end{array}\right)\mapsto\Lambda\left(\begin{array}{c}\bar{x} \\ \bar{y}\end{array}\right)$$
into 
$$\left(\begin{array}{c}\bar{\xi} \\ \bar{\eta}\end{array}\right)\mapsto\left(\begin{array}{c}\lambda^{-1}\bar{\xi} \\ \lambda\bar{\eta}\end{array}\right).$$
In the $\bar{\xi}\text{-}\bar{\eta}$ coordinate, we denote $$A=(0,\eta)\quad \text{and}\quad B= (\xi,0).$$

Let $\dots, y_{-1},\; y_0,\; y_1,\dots$ be a periodic orbit with minimal period $(n+n_0+m_0)q-1$, rotation number $$\frac{(n+n_0+m_0)p}{(n+n_0+m_0)q-1},$$ and $y_0\in U_{z_0} $, which is a small neighborhood of    $z_0$  such that 
$$f^{m_0q}(U_{z_0})\subset U\quad {\rm and} \quad f^{-n_0q+1}(U_{z_0})\subset U.$$
We choose $n\in\mathbb{N}$ to be sufficiently large.

 Let us denote $$A'=f^{m_0q}(y_0),\quad B'=f^{-n_0q+1}(y_0).$$
 Then in the coordinates $\bar{\xi}\text{-}\bar{\eta}$, they become 
 
 $$A'=\left(\begin{array}{c}\delta_A \\ \eta+\delta_A'\end{array}\right) \quad {\rm and}\quad B'=\left(\begin{array}{c}\xi+\delta_B \\ \delta_B'\end{array}\right).$$
 Here the $\delta$'s are  small numbers to be determined.
 
 By the periodicity of $y_0$, we have that 
 $$\left(\begin{array}{c}\lambda^{-n}\delta_A \\ \lambda^n(\eta+\delta_A')\end{array}\right)
 =\left(\begin{array}{c}\xi+\delta_B\\ \delta_B'\end{array}\right),$$
 and
 $$\left(\begin{array}{c}\delta_A \\ \delta_A'\end{array}\right)=\left(\begin{array}{cc}a&b \\ c&d\end{array}\right)\left(\begin{array}{c}\delta_B\\ \delta_B'\end{array}\right)+\mathcal{O}(\delta_{B}^{3/2}+(\delta_B')^{3/2}),$$
 where the $2\times2$ matrix on the right side is the linear part of the global map $R_{\theta}\circ\Phi^{-1}\circ f^{(n_0+m_0)q-1}\circ \Phi\circ R_{-\theta}$ at the point $B$ (the global map is of $C^{1,1/2}$).
 Due to the transversal intersections between the stable and unstable manifolds at points $A$ and $B$, we have $a\neq0$.
Therefore,\[ \begin{cases}\delta_A=\xi\lambda^n+\mathcal{O}(\lambda^{2n}),\quad\delta'_A=\frac{c\xi-\eta}{a}\lambda^n+\mathcal{O}(\lambda^{3n/2}),\\
 \delta_B'=\eta\lambda^n+\mathcal{O}(\lambda^{2n}),\quad \delta_B=\frac{\xi-b\eta}{a}\lambda^n+\mathcal{O}(\lambda^{3n/2}).\end{cases}\]
 
 \vspace{5 pt}
 
 Now, let $n_1=\lfloor{n/2}\rfloor$ and $n_2=n-n_1$. 
 
 In the $\bar{\xi}\text{-}\bar{\eta}$ coordinates, for $i=0,1,\dots,n_1$, the difference between the images of the  points $f^{m_0q+iq}(y_0)$ and $f^{m_0q+iq}(z_0)$ is
$$ \left(\begin{array}{c}\lambda^{-i}\delta_A \\(\eta+\delta_A')\lambda^i\end{array}\right)-\left(\begin{array}{c}0 \\ \eta\lambda^i\end{array}\right)=\left(\begin{array}{c}\xi\lambda^{n-i} \\ 0\end{array}\right)+\left(\begin{array}{c}0\\\frac{c\xi-\eta}{a}\lambda^{n+i}\end{array}\right)+\mathcal{O}(\lambda^{3n/2})$$
and for $j=0,1,\dots, n_2$, the difference between the images of the points $f^{-n_0q+1-jq}(y_0)$ and $f^{-n_0q+1-jq}(z_0)$ is
$$\left(\begin{array}{c}\lambda^j(\xi+\delta_B) \\\delta_B'\lambda^{-j}\end{array}\right)-\left(\begin{array}{c}\xi\lambda^j \\0\end{array}\right)=\left(\begin{array}{c}0 \\ \eta\lambda^{n-j}\end{array}\right)+\left(\begin{array}{c}\frac{\xi-b\eta}{a}\lambda^{n+j}\\0\end{array}\right)+\mathcal{O}(\lambda^{3n/2}).\\$$

\vspace{5 pt}

 Back to the coordinate $(\bar{x},\bar{y})$.

For $i=0,1,\dots, n_1$, along the {stable direction}, the differences between the periodic orbit and the homoclinic orbit is
\[\begin{split}&\left(\begin{array}{cc}\cos\theta&-\sin\theta\\\sin\theta&\cos\theta\end{array}\right)\left(\begin{array}{c}\xi\lambda^{n-i}\\\frac{c\xi-\eta}{a}\lambda^{n+i}\end{array}\right)+\mathcal{O}(\lambda^{3n/2})\\
&=\left(\begin{array}{c}\cos\theta \\\sin\theta\end{array}\right)\xi\lambda^{n-i}+\left(\begin{array}{c}-\sin\theta\\\cos\theta\end{array}\right)\frac{c\xi-\eta}{a}\lambda^{n+i}+\mathcal{O}(\lambda^{3n/2}).\end{split}\]
For $j=0,1,\dots,n_2$, along {the unstable direction}
\[\begin{split}&\left(\begin{array}{cc}\cos\theta&-\sin\theta\\\sin\theta&\cos\theta\end{array}\right)\left(\begin{array}{c}\frac{\xi-b\eta}{a}\lambda^{n+j}\\\eta\lambda^{n-j}\end{array}\right)+\mathcal{O}(\lambda^{3n/2})\\
&=\left(\begin{array}{c}-\sin\theta \\\cos\theta\end{array}\right)\eta\lambda^{n-j}+\left(\begin{array}{c}\cos\theta\\\sin\theta\end{array}\right)\frac{\xi-b\eta}{a}\lambda^{n+j}+\mathcal{O}(\lambda^{3n/2}).\end{split}\]

\vspace{5 pt}

Back to the original coordinate $(s,r)$. 

For the  orbits $\dots, y_{-1},\;y_0,\;y_1,\dots$ and $\dots,z_{-1},z_0,z_1,\dots$, by Lemma~\ref{normal-form}, we have that 
for $i=0,1,\dots, n_1$, 
\begin{equation}\label{stable-split}z_{m_0q+iq}-y_{m_0q+iq}=-\left(\begin{array}{c}\cos\theta \\\sin\theta\end{array}\right)\xi\lambda^{n-i}+\mathcal{O}(\lambda^{i/2} \lambda^{n-i}),\end{equation}
and for $j=0,1,\dots,n_2$,
\begin{equation}\label{unstable-split}z_{-n_0q+1-jq}-y_{-n_0q+1-jq}=-\left(\begin{array}{c}-\sin\theta\\\cos\theta\end{array}\right)\eta\lambda^{n-j}+\mathcal{O}(\lambda^{j/2} \lambda^{n-j}).\end{equation}

Now consider the quantity 
$$I=\sum_{i=-(n_0+n_2)q+1}^{(m_0+n_1)q-1}h(z_i,z_{i+1})-h(y_i,y_{i+1}).$$
We split it into three parts:
The first part corresponds to the sum far away from the minimal periodic orbit $X_{p/q}$:
$$\mathcal{I}_0=\sum_{i=-n_0q+1}^{m_0q-1}h(z_i,z_{i+1})-h(y_i,y_{i+1}),$$
the second part is along the unstable manifold:
$$\mathcal{I}_1=\sum_{i=-(n_0+n_2)q+1}^{-n_0q}h(z_i,z_{i+1})-h(y_i,y_{i+1}),$$
and the third part is close to  the stable manifold:
$$\mathcal{I}_2=\sum_{i=m_0q}^{(m_0+n_1q)-1}h(z_i,z_{i+1})-h(y_i,y_{i+1}).$$

\vspace{10 pt}

\begin{itemize}
\item
Since along the periodic orbit $y_i$, $i\in\mathbb{Z}$, 
\begin{equation}\label{periodic-null}\partial_2h(y_i,y_{i+1})+\partial_1 h(y_{i+1},y_{i+2})=0,\quad i\in\mathbb{Z}\end{equation}
we  have that 
\begin{equation}\label{I0}\begin{split}\mathcal{I}_0&=\sum_{i=-n_0q+1}^{m_0q-1}h(z_i,z_{i+1})-h(y_i,y_{i+1}).\\
&=\partial_1h(z_{-n_0q+1},z_{ -n_0q+2})(y_{-n_0q+1}-z_{-n_0q+1})\\
&\quad +\sum_{i=-n_0q+1}^{m_0q-1}\frac{1}{2}D^2h(y_i,y_{i+1})(z_i-y_i)^2\\
&\quad+\partial_2h(z_{m_0q-1}, z_{m_0q})(y_{m_0q}-z_{m_0q})+\mathcal{O}((m_0q+n_0q-1)\lambda^{3n})\\
&=\partial_1h(y_{-n_0q+1},y_{ -n_0q+2})(z_{-n_0q+1}-y_{-n_0q+1})\\
&\quad+\partial_2h(y_{m_0q-1}, y_{m_0q})(z_{m_0q}-y_{m_0q})+\mathcal{O}(\lambda^{2n}).
\end{split}\end{equation}

\vspace{10 pt}

\item Next, we  consider $\mathcal{I}_1$, which is the sum of  the terms along the unstable manifold.

For $j=1,\dots, n_2$, let us denote 
$$\tilde{z}_k^j=z_{-n_0q-jq+1+k},\quad  \tilde{y}_k^j=y_{-n_0q-jq+1+k},\quad k=0,\dots, q,$$
and 
$$I_j:=\sum_{k=0}^{q-1}h(\tilde{z}_k^j,\tilde{z}_{k+1}^j)-h(\tilde{y}^j_k,\tilde{y}^j_{k+1}).$$
Clearly,
$\mathcal{I}_1=\sum_{j=1}^{n_2}I_j$.
We continue to split it into two other sums:
$$\mathcal{I}_1=\sum_{j=1}^{n_2/2-1}I_j+\sum_{j=n_2}^{n_2/2}I_j.$$
Let us first consider the cases $j=n_2/2,\dots,n_2$. 

By \eqref{unstable-split} and  \eqref{periodic-null}, we have
\[\begin{split}I_j&=\partial_1h(\tilde{y}^j_0,\tilde{y}^j_1)(\tilde{z}^j_0-\tilde{y}^j_0)+\partial_2h(\tilde{y}^j_{q-1},\tilde{y}^j_{q})(\tilde{z}^j_{q}-\tilde{y}^j_{q})\\
&\quad +\frac{1}{2}\sum_{k=0}^{q-1}\left(\begin{array}{c}\tilde{y}^j_k-\tilde{z}^j_k\\\tilde{y}^j_{k+1}-\tilde{z}^j_{k+1}\end{array}\right)^TD^2h(\tilde{y}_k^j,\tilde{y}_{k+1}^j)\left(\begin{array}{c}\tilde{y}^j_k-\tilde{z}^j_k\\\tilde{y}^j_{k+1}-\tilde{z}^j_{k+1}\end{array}\right)+\mathcal{O}(\lambda^{3(n-i)})
\end{split}\]
where 
\[\begin{split}D^2h(\tilde{y}_k^j,\tilde{y}_{k+1}^j)&=\left(\begin{array}{cc}\partial_{11}h(\tilde{y}^j_{k},\tilde{y}^j_{k+1}) & \partial_{12}h(\tilde{y}^j_{k},\tilde{y}^j_{k+1}) \\\partial_{21}h(\tilde{y}^j_{k},\tilde{y}^j_{k+1}) & \partial_{22}h(\tilde{y}^j_{k},\tilde{y}^j_{k+1})\end{array}\right)\\
&=\left(\begin{array}{cc}\partial_{11}h(x_{k+1},x_{k+2}) & \partial_{12}h(x_{k+1},x_{k+2}) \\\partial_{21}h(x_{k+1},x_{k+2}) & \partial_{22}h(x_{k+1},x_{k+2})\end{array}\right)+\mathcal{O}(\lambda^{n/4}).\end{split}\]
Here we have use that fact that for $j=n_2/2,\dots,n_2$,   $\tilde{y}_k^j$ are  at least \mbox{$\mathcal{O}(\lambda^{n/4})$-close} to $x_{k+1}$, $k=0,\dots,q$.

From \eqref{unstable-split} we know that
$$\tilde{z}_{k}^j-\tilde{y}_k^j=\lambda^{n-j}\prod_{i=0}^{k-1} Df(x_{i+1})\left(\begin{array}{c}\sin\theta\\-\cos\theta\end{array}\right)\eta+\mathcal{O}(\lambda^{n-j+\frac{n}{8}}),\quad k=1,\dots,q.$$
Let  us denote $Z_0^+=\sin\theta$, and 
\begin{equation}\label{Z_K+}Z_k^+=\pi_{1}\Big[\prod_{i=0}^{k-1} Df(x_{i+1})\left(\begin{array}{c}\sin\theta\\-\cos\theta\end{array}\right)\Big],\quad k=1,\dots,q\end{equation}
where $\pi_1$ is the projection on the first coordinate.
Denote  \begin{equation}\label{Z+}\mathcal{Z}_+=(Z_0^+,Z_1^+,\dots,Z_q^+).\end{equation}
Then we have
\[\begin{split}&\sum_{k=1}^{q-1}\left(\begin{array}{c}\tilde{y}^j_k-\tilde{z}^j_k\\\tilde{y}^j_{k+1}-\tilde{z}^j_{k+1}\end{array}\right)^TD^2h(\tilde{y}_k^i,\tilde{y}_{k+1}^i)\left(\begin{array}{c}\tilde{y}^j_k-\tilde{z}^j_k\\\tilde{y}^j_{k+1}-\tilde{z}^j_{k+1}\end{array}\right)\\
&=\mathcal{Z}_+\mathbb{W}(X_{p/q})\mathcal{Z}_+^T\eta^2\lambda^{2(n-j)}+\mathcal{O}(\lambda^{2(n-j)+\frac{n}{8}}),\end{split}\]
where
\begin{equation}\label{matrix}\mathbb{W}(X_{p/q})=\left(\begin{array}{ccccc}\eta_1 & \sigma_1 & 0 & \dots & 0 \\\sigma_1 & \eta_2 & \sigma_2 & \dots & 0 \\ \vdots&   \ddots & \ddots &\ddots&\vdots \\&  & \sigma_{q-1}&\eta_{q}&\sigma_q\\0 &0& \dots  & \sigma_{q} & \eta_{q+1}\end{array}\right)_{(q+1)\times (q+1)}\end{equation}
with $\eta_1=\partial_{11}h(x_1,x_2)$, $\eta_{q+1}=\partial_{22}h(x_{0},x_1)$,
$$\eta_i=\partial_{22}h(x_{i-1},x_i)+\partial_{11}h(x_i,x_{i+1}),\quad i=2,\dots,q$$
and   $$\sigma_{i}=\partial_{12}h(x_{i},x_{i+1}), \quad i=1,\dots,q.$$
Then for $j=n_1/2,\dots,n_2$, we have
\[I_j=\partial_1h(\tilde{y}^j_0,\tilde{y}^j_1)(\tilde{z}^j_0-\tilde{y}^j_0)+\partial_2h(\tilde{y}^j_{q-1},\tilde{y}^j_{q})(\tilde{z}^j_{q}-\tilde{y}^j_{q})+C_{q+}\eta^2\lambda^{2(n-j)}+\mathcal{O}(\lambda^{2(n-j)+\frac{n}{8}})\]
where \begin{equation}C_{q+}=\frac{1}{2}\mathcal{Z}_+\mathbb{W}(X_{p/q})\mathcal{Z}_+^T.\end{equation}
And for $j=1,\dots, n_2/2-1$, by \eqref{unstable-split}, we have
$$I_j=\partial_1h(\tilde{y}^j_0,\tilde{y}^j_1)(\tilde{z}^j_0-\tilde{y}^j_0)+\partial_2h(\tilde{y}^j_{q-1},\tilde{y}^j_{q})(\tilde{z}^j_{q}-\tilde{y}^j_{q})+\mathcal{O}(\lambda^{2(n-j))}).$$
Hence
\begin{equation}\label{I1}\begin{split}\mathcal{I}_1&=\sum_{j=1}^{n_2}I_j=\sum_{j=1}^{n_2/2-1}I_j+\sum_{j=n_2/2}^{n_2}I_j\\
&=\partial_1h(y_{-(n_0+n_2)q+1},y_{-(n_0+n_2)q+2})(z_{-(n_0+n_2)q+1}-y_{-(n_0+n_2)q+1})\\
&\quad+\partial_2h(y_{-n_0q+1},y_{-n_0q+2})(z_{n_0q+1}-y_{-n_0q+1})\\
&\quad+C_{q+}\eta^2\frac{\lambda^{2(n-n_2)}}{1-\lambda^2}+\mathcal{O}(\lambda^{9n/8}).\end{split}\end{equation}

\vspace{10 pt}

\item Now we deal with $\mathcal{I}_2$, the sum of  the terms along the stable manifold.

For $i=1,\dots, n_1$, let us denote 
$$\bar{z}_k^i=z_{m_0q+(i-1)q+k},\quad \bar{y}_{k}^i=y_{m_0q+{i-1}q+k},\quad k=0,\dots, q,$$
and
$$\bar{I}_i=\sum_{k=0}^{q-1}h(\bar{z}_k^i,\bar{z}_{k+1}^i)-h(\bar{y}_k^i,\bar{y}_{k+1}^i).$$
Clearly, $\mathcal{I}_2=\sum_{i=1}^{n_1}\bar{I}_i$. We split it into two parts:
$$\mathcal{I}_2=\sum_{i=1}^{n_1/-1}\bar{I}_i+\sum_{i=n_2/2}^{n_2}\bar{I}_i.\\$$

\vspace{5 pt}

First, consider the cases  $i=n_1/2,\dots,n_1$.
By \eqref{periodic-null}, we have that

\[\begin{split}\bar{I}_i&=\partial_{1}h(\bar{y}_0^i,\bar{y}_1^i)(\bar{z}_0^i-\bar{y}_0^i)+\partial_2h(\bar{y}_{q-1}^i,\bar{y}_q^i)(\bar{z}_q^i-\bar{y}_q^i)\\
 &\quad+\frac{1}{2}\sum_{k=0}^{q-1}\left(\begin{array}{c}\bar{y}^i_k-\tilde{z}^i_k\\\bar{y}^i_{k+1}-\bar{z}^i_{k+1}\end{array}\right)^TD^2h(\bar{y}_k^i,\bar{y}_{k+1}^i)\left(\begin{array}{c}\bar{y}^i_k-\bar{z}^i_k\\\bar{y}^i_{k+1}-\bar{z}^i_{k+1}\end{array}\right)+\mathcal{O}(\lambda^{3(n-i)}).\end{split}\]
 
Due to \eqref{stable-split}, we have
$$\bar{z}_0^i-\bar{y}_0^i=-\lambda^{n-i+1}\left(\begin{array}{c}\cos\theta\\\sin\theta\end{array}\right)\xi+\mathcal{O}(\lambda^{n-i+\frac{n}{8}}),$$
and for $k=1,\dots,q$,
$$\bar{z}_k^i-\bar{y}_k^i=-\lambda^{n-i+1}\prod_{l=0}^{k-1}Df(x_{l+1})\left(\begin{array}{c}\cos\theta \\\sin\theta\end{array}\right)\xi+\mathcal{O}(\lambda^{n-i+\frac{n}{8}}).$$
Denote $Z_0^-=-\cos\theta$, 
\begin{equation}\label{Z_K-}Z_k^-=\pi_1\Big[-\prod_{l=0}^{k-1}Df(x_{l+1})\left(\begin{array}{c}\cos\theta \\\sin\theta\end{array}\right)\Big], k=1,\dots,q,\end{equation}
and
\begin{equation}\label{Z-}\mathcal{Z}_-=(Z_0^-,\dots,Z_{q}^-).\end{equation}
Then
\[\begin{split}\bar{I}_i&=\partial_{1}h(\bar{y}_0^i,\bar{y}_1^i)(\bar{z}_0^i-\bar{y}_0^i)+\partial_2h(\bar{y}_{q-1}^i,\bar{y}_q^i)(\bar{z}_q^i-\bar{y}_q^i)\\
 &\quad+\frac{1}{2}\mathcal{Z}_-\mathbb{W}(X_{p/q})\mathcal{Z}_-^T\xi^2\lambda^{2(n-i+1)}+\mathcal{O}(\lambda^{2(n-i)+\frac{n}{8}}),\end{split}\]
where $\mathbb{W}(X_{p/q})$ is defined in \eqref{matrix}.
Moreover, for $i=1,\dots,n_1/2-1$ we have 
$$\partial_{1}h(\bar{y}_0^i,\bar{y}_1^i)(\bar{z}_0^i-\bar{y}_0^i)+\partial_2h(\bar{y}_{q-1}^i,\bar{y}_q^i)(\bar{z}_q^i-\bar{y}_q^i)+\mathcal{O}(\lambda^{2(n-i)}).$$
Then
\begin{equation}\label{I2}\begin{split}\mathcal{I}_2&=\sum_{i=1}^{n_1/2-1}\bar{I}_i+\sum_{i=n_1/2}^{n_1}\bar{I}_i\\&=\partial_1h(y_{m_0q},y_{m_0q+1})(z_{m_0q}-y_{m_0q})\\
&\quad+\partial_2h(y_{(m_0+n_1)q-1},y_{(m_0+n_1)q})(z_{(m_0+n_1)q}-y_{(m_0+n_1)q})\\
&\quad+C_{q-}\xi^2\frac{\lambda^{2(n-n_1+1)}}{1-\lambda^2}+\mathcal{O}(\lambda^{9n/8})
\end{split}\end{equation}
where 
\begin{equation} C_{q-}=\frac{1}{2}\mathcal{Z}_-\mathbb{W}(X_{p/q})\mathcal{Z}_-^T.\end{equation}
\end{itemize}

\vspace{10 pt}

To sum up \eqref{I0}, \eqref{I1} and \eqref{I2}, we have 
\[\begin{split}I&=\mathcal{I}_0+\mathcal{I}_1+\mathcal{I}_2\\
&=\partial_1h(y_{-(n_0+n_2)q+1},y_{-(n_0+n_2)q+2})(z_{-(n_0+n_2)q+1}-y_{-(n_0+n_2)q+1})\\
&\quad +\partial_2h(y_{(m_0+n_1)q-1},y_{(m_0+n_1)q})(z_{(m_0+n_1)q}-y_{(m_0+n_1)q})\\
&\quad +C_{q+}\eta^2\frac{\lambda^{2(n-n_2)}}{1-\lambda^2} +C_{q-}\xi^2\frac{\lambda^{2(n-n_1+1)}}{1-\lambda^2}+\mathcal{O}(\lambda^{9n/8}).\end{split}\]

\vspace{10 pt}

Now we consider the tail:
\[I_+=\sum_{i=m_0q+n_1q}^{+\infty}h(z_i,z_{i+1})-h(x_{i+1},x_{i+2}).\]
Since along the periodic orbit $X_{p/q}$, 
\begin{equation}\label{periodic-null2}\partial_2h(x_i,x_{i+1})+\partial_1h(x_{i+1},x_i)=0,\quad i\in\mathbb{Z},\end{equation}
we have 
\[\begin{split}I_+&=\partial_1h(x_1,x_2)(z_{m_0q+n_1q}-x_1)\\
&\quad +\frac{1}{2}\sum_{i=m_0q+n_1q}^{+\infty}\left(\begin{array}{c}z_i-x_{i+1}\\z_{i+1}-x_{i+2}\end{array}\right)^TD^2h(x_{i+1},x_{i+2})\left(\begin{array}{c}z_i-x_{i+1}\\z_{i+1}-x_{i+2}\end{array}\right)+\mathcal{O}(\lambda^{3n/2}).\end{split}\]
Since $x_{m_0q+n_1q+1}=x_1$ and
$$z_{m_0q+n_1q}-x_1=\lambda^{n_1}\left(\begin{array}{c}-\sin\theta \\\cos\theta\end{array}\right)\eta+\mathcal{O}(\lambda^{3n_1/2}),$$
by the same calculation of $\mathcal{I}_1$, we obtain
\begin{equation}\label{I+}\begin{split}I_+&=\partial_1h(x_1,x_2)(z_{m_0q+n_1q}-x_1)+\frac{\lambda^{2n_1}}{1-\lambda^2}\frac{\eta^2}{2}\mathcal{Z}_+\mathbb{W}(X_{p/q})\mathcal{Z}_+^T+\mathcal{O}(\lambda^{5n/4})\\
&=\partial_1h(x_1,x_2)(z_{m_0q+n_1q}-x_1)+\frac{\lambda^{2n_1}}{1-\lambda^2}C_{q+}\eta^2+\mathcal{O}(\lambda^{5n/4}).\end{split}\end{equation}
Similarly,
we have
\begin{equation}\label{I-}\begin{split} I_-&=\sum_{i=-n_0q+1-n_2q}^{-\infty}h(z_{i-1},z_i)-h(x_{i-1},x_i)\\
&=\partial_2h(x_0,x_1)(z_{-(m_0+n_2)q+1}-x_1)+C_{q-}\xi^2\frac{\lambda^{2(n_2+1)}}{1-\lambda^2}+\mathcal{O}(\lambda^{5n/4}).
\end{split}\end{equation}
Then \begin{equation}\begin{split}\mathbb{I}&:=I+I_-+I_+\\
&=\partial_1h(y_{-(n_0+n_2)q+1},y_{ -(n_0+n_2)q+2})(z_{-(n_0+n_2)q+1}-y_{-n_0q-n_2q+1})\\
 &\quad+\partial_2h(y_{(m_0+n_1)q-1}, y_{(m_0+n_1)q})(z_{(m_0+n_1)q}-y_{(m_0+n_1)q})\\
 &\quad+\partial_2h(x_0,x_1)(z_{-(m_0+n_2)q+1}-x_1)+\partial_1h(x_1,x_2)(z_{m_0q+n_1q}-x_1)\\
 &\quad+2C_{q+}\eta^2\frac{\lambda^{2n_1}}{1-\lambda^2}+2C_{q-}\xi^2\frac{\lambda^{2(n_2+1)}}{1-\lambda^2}+\mathcal{O}(\lambda^{9n/8}).\end{split}\end{equation}
 Since $y_{(m_0+n_1)q}=y_{-(n_0+n_2)q+1}$, by \eqref{periodic-null}, 
we have
\[\begin{split}&\partial_2h(y_{(m_0+n_1)q-1}, y_{(m_0+n_1)q})(z_{(m_0+n_1)q}-y_{(m_0+n_1)q})\\
 &+\partial_1h(y_{-(n_0q+n_2)q+1},y_{ -(n_0+n_2)q+2})(z_{-(n_0+n_2)q+1}-y_{-(n_0+n_2)q+1})\\
& =\partial_2h(y_{(m_0+n_1)q-1}, y_{(m_0+n_1)q})(z_{(m_0+n_1)q}-z_{-(n_0+n_2)q+1}).\end{split}\]
Notice that   
  \[\begin{split}y_{(m_0+n_1)q}-x_1&=y_{(m_0+n_1)q}-z_{(m_0+n_1)q}+z_{(m_0+n_1)q}-x_1\\
  &=\left(\begin{array}{c}\cos\theta\\\sin\theta\end{array}\right)\xi\lambda^{n-n_1}+\left(\begin{array}{c}-\sin\theta\\\cos\theta\end{array}\right)\eta\lambda^{n_1}+\mathcal{O}(\lambda^{3n/4}),\end{split}\]
  
  $$y_{(m_0+n_1)q-1}-x_0=\left(\begin{array}{c}a'(\xi\lambda^{n_2}\cos\theta-\eta\lambda^{n_1}\sin\theta)+b'( \xi\lambda^{n_2}\sin\theta+\eta\lambda^{n_1}\cos\theta)\\\gamma_0'\end{array}\right)+\mathcal{O}(\lambda^{3n/4}),$$
  
  $$z_{(m_0+n_1)q}-z_{-(n_0+n_2)q+1}=\left(\begin{array}{c}-\eta\lambda^{n_1}\sin\theta-\xi\lambda^{n_2}\cos\theta \\ \gamma_1'\end{array}\right)+\mathcal{O}(\lambda^{3n/4}),$$
  where $a'$ and $b'$ are from the expression $$Df^{-1}(x_1)=\left(\begin{array}{cc}a'&b'\\**&*\end{array}\right),$$
  with
  $$a'=\frac{-\partial_{22}h(x_0,x_1)}{\partial_{12}h(x_0,x_1)},\quad b'=\frac{1}{\partial_{12}h(x_0,x_1)},$$
  (here we have use \eqref{derivative2}).
Thus we have 
\[\begin{split}&\partial_2h(y_{(m_0+n_1)q-1}, y_{(m_0+n_1)q})\\&=\partial_2 h(x_0,x_1)
+\partial_{12}h(x_0,x_1)(y_{(m_0+n_1)q-1}-x_0)+\partial_{22}h(x_0,x_1)(y_{(m_0+n_1)q}-x_1)+\mathcal{O}(\lambda^{n})\\
&=\partial_2 h(x_0,x_1)+\partial_{22}h(x_0,x_1)\big[\xi\lambda^{n_2}\cos\theta-\eta\lambda^{n_1}\sin\theta\big]+\mathcal{O}(\lambda^{3n/4})\\
&\quad+\partial_{12}h(x_0,x_1)\Big[-\frac{\partial_{22}h(x_0,x_1)}{\partial_{12}h(x_0,x_1)}(\xi\lambda^{n_2}\cos\theta-\eta\lambda^{n_1}\sin\theta)+\frac{1}{\partial_{12}h(x_0,x_1)}(\xi\lambda^{n_2}\sin\theta+\eta\lambda^{n_1}\cos\theta)\Big]\\
&=\partial_2h(x_0,x_1)+(\xi\lambda^{n_2}\sin\theta+\eta\lambda^{n_1}\cos\theta)+\mathcal{O}(\lambda^{3n/4}).\end{split}\]
Therefore,
\begin{equation}\label{cross}\begin{split}&\partial_2h(y_{(m_0+n_1)q-1}, y_{(m_0+n_1)q})(z_{(m_0+n_1)q}-z_{-(n_0+n_2)q+1})\\
&=\partial_2 h(x_0,x_1)(z_{(m_0+n_1)q}-z_{-(n_0+n_2)q+1})\\
&\quad+(\xi\lambda^{n_2}\sin\theta+\eta\lambda^{n_1}\cos\theta)(-\xi\lambda^{n_2}\cos\theta-\eta\lambda^{n_1}\sin\theta)+\mathcal{O}(\lambda^{5n/4})\\
&=\partial_2 h(x_0,x_1)(z_{(m_0+n_1)q}-z_{-(n_0+n_2)q+1})\\
&\quad-\sin\theta\cos\theta (\xi^2\lambda^{2n_2}+\eta^2\lambda^{2n_1})-\xi\eta\lambda^n+\mathcal{O}(\lambda^{5n/4}).\end{split}\end{equation}
Due to  \eqref{periodic-null2},
\[\begin{split}&\partial_2h(x_0,x_1)(z_{m_0+n_1q}-z_{(n_0+n_2)q+1})\\
&+\partial_2h(x_0,x_1)(z_{-(m_0+n_2)q+1}-x_1)+\partial_1h(x_1,x_2)(z_{m_0q+n_1q}-x_1)=0,\end{split}\]
hence we have,
\begin{equation}\label{final}\begin{split}\mathbb{I}&=2C_{q+}\eta^2\frac{\lambda^{2n_1}}{1-\lambda^2}+2C_{q-}\xi^2\frac{\lambda^{2(n_2+1)}}{1-\lambda^2}\\
&\quad-\sin\theta\cos\theta( \xi^2\lambda^{2n_2}+\eta^2\lambda^{2n_1})-\xi\eta\lambda^{n}+\mathcal{O}(\lambda^{9n/8}).\end{split} \end{equation}

\vspace{10 pt}

If $n$ is even, then we have
{
\be\label{evenC}
\begin{split}\mathbb{I}&=\Big(\frac{2C_{q+}\eta^2}{1-\lambda^2}+\frac{2C_{q-}\lambda^2\xi^2}{1-\lambda^2}-\xi^2\sin\theta\cos\theta-\eta^2\sin\theta\cos\theta-\xi\eta\Big)\frac{\lambda^{n_0+m_0+n}}{\lambda^{-m_0-n_0}}\\
&\quad+\mathcal{O}(\lambda^{9n/8})\\
&:=C_{p,q}\lambda^{n_0+m_0+n}+\mathcal{O}(\lambda^{9n/8}),\end{split}
\ee}
and 
if $n$ is odd, that is  $n=2n_1+1$, then
{\be \label{oddC}
\begin{split}\mathbb{I}&=\Big(\frac{2C_{q+}\eta^2}{\lambda(1-\lambda^2)}+\frac{2C_{q-}\xi^2\lambda^3}{1-\lambda^2}-\xi^2\lambda\sin\theta\cos\theta-\eta^2\lambda^{-1}\sin\theta\cos\theta-\xi\eta\Big)\frac{\lambda^{m_0+n_0+n}}{\lambda^{-m_0-n_0}}\\
&\quad+\mathcal{O}(\lambda^{9n/8})\\
&:=C_{p,q}'\lambda^{n_0+m_0+n}+\mathcal{O}(\lambda^{9n/8}).
\end{split}
\ee }

\vspace{10 pt}

Summarizing, the proof of the assertion follows by denoting $N=n_0+m_0+n$.
\end{proof}

\bigskip

\begin{remark} The constants in \eqref{final} are independent of the choice of the base point where we apply the normal form Lemma \ref{normal-form}. E.g, if we choose $x_2$ as the base point,  then in  \eqref{I1} and \eqref{I+}, the terms of the order $\lambda^{2n_1}$ become
$$C_{q+}\frac{\eta^2\lambda^{2n_1}}{1-\lambda^2}-\frac{1}{2}\Big[h_{11}(x_1,x_2)(Z_0^+)^2+2h_{12}(x_1,x_2)Z_0^+Z_1^++h_{22}(x_1,x_2)(Z_1^+)^2\Big]\eta^2\lambda^{2n_1}, $$
and the terms of order $\lambda^{2(n_2+1)}$ in \eqref{I2} and \eqref{I-} turn into
$$C_{q-}\frac{\xi^2\lambda^{2(n_2+1)}}{1-\lambda^2}+\frac{1}{2}\Big[h_{11}(x_1,x_2)(Z_0^-)^2+2h_{12}(x_1,x_2)Z_0^-Z_1^-+h_{22}(x_1,x_2)(Z_1^-)^2\Big]\xi^2\lambda^{2n_2}.$$
Those in \eqref{cross} become
\[\begin{split}&\Big[\partial_{12}h(x_1,x_2)(x_1,x_2)(Z_0^-Z_1^+-Z_0^+Z_1^-)+\partial_{22}h(x_1,x_2)(Z_1^-Z_1^+-Z_1^-Z_1^+)\Big]\xi\eta\lambda^n\\
&+\Big(\partial_{12}h(x_1,x_2)Z_0^+Z_1^++\partial_{22}h(x_1,x_2)(Z_1^+)^2\Big)\eta^2\lambda^{2n_1}\\
&+\Big(-\partial_{12}h(x_1,x_2)Z_0^-Z_1^--\partial_{22}h(x_1,x_2)(Z_1^-)^2\Big)\xi^2\lambda^{2n_2}.\end{split}\]
Then adding them up, using \eqref{derivative1}, \eqref{derivative2}, \eqref{Z_K+} and \eqref{Z_K-}, we have exactly \eqref{final}.\\
\end{remark}

\medskip

\begin{lemma}When $N$ is sufficiently  large, the periodic orbit obtained {in Theorem \ref{lemma-apr}}, is the one with the maximal perimeter, i.e., an Aubry-Mather periodic orbit.\\
\end{lemma}

\begin{proof}
Let $X_{Np,Nq-1}'$ denote the periodic orbit with minimal period $Nq-1$, rotation number $\frac{Np}{Nq-1}$ and the  maximal perimeter (minimal action). Then the distance $d(z_0,X_{Np,Nq-1}')$ tends zero   as $N$ tends to $+\infty$.  By hyperbolicity,  there exists a neighborhood $U$ of $z_0$ which contains exactly one periodic orbit with minimal period $Nq-1$ and  rotation number  $\frac{Np}{Nq-1}$. Therefore $X_{Np,Nq-1}$ and $X'_{Np,Nq-1}$ coincide when $N$ is large enough. 
\end{proof}

\medskip

\brm \label{rm:general-case1}
It seems that Theorem \ref{lemma-apr} holds true in general, namely, 
suppose we have a hyperbolic periodic orbit of a billiard map and a transverse homoclinic orbit related to it. 
Then, the difference of perimeters should satisfy the estimate from Theorem \ref{lemma-apr}. \\
\erm

\begin{lemma} \label{log-limit}If the constants $C_{p,q}$ and $C_{p,q}'$ in Theorem \ref{lemma-apr} are not zero, then
$$\lim_{N\to+\infty}\frac{1}{N}\log|L_{Np, Nq-1}-N\cdot L_{p,q}+B_{p/q}(z_0)|=\log\lambda.\\$$
\end{lemma}

\vspace{10 pt}

{In the remaining part of this section, we want to prove that these constants are generically non-zero.}
From now on, we use the notation $C_{p,q}(f)$, $C_{p,q}'(f)$, $\lambda(f)$, $\xi(f)$, etc.  to indicate explicitly the dependence on $f$.\\

\begin{lemma}\label{lemma-parametrized} There exist $\bar{\epsilon}_1>0$, $\bar{\epsilon}_2>0$ and a family of 
billiard maps $f_{\epsilon_1,\epsilon_2}$ parametrized by $\epsilon_1\in[-\bar{\epsilon}_1,\bar{\epsilon}_1]$ and 
$\epsilon_2\in[-\bar{\epsilon}_2,\bar{\epsilon}_2]$ such that
$$
f_{0,0}=f, \quad \lambda(f_{\epsilon_1,\epsilon_2})=\lambda(f),\quad\theta(f_{\epsilon_1,\epsilon_2})=\theta(f),
$$
and 
$$
\frac{d}{d\epsilon_1}\xi(f_{\epsilon_1,\epsilon_2})\neq0,\quad \frac{d}{d\epsilon_2}\eta(f_{\epsilon_1,\epsilon_2})\neq0.
$$
Moreover,
$$\|f_{\epsilon_1,\epsilon_2}-f\|_{C^{\tau}}\to0, \quad\text{as}\quad \epsilon_1\to0,\;\;\epsilon_2\to0.$$
\end{lemma}

\vspace{10 pt}

\begin{proof}  
Let us denote$$s_{i}'=\pi_1\big(f^{i}(z_0)\big), \quad i=-2,-1,0,1,2.$$
Because of the graph property of the orbit $X_{p/q+}$,  for $i=-2,-1,0,1,2$,  there exist  $\gamma_i^-<0$, 
$\gamma_i^+>0$ and  functions $\varphi_i$ such that  the following holds:

\begin{enumerate}
\item $\big\{(s,r): s\in[s_i'+\gamma_i^-,s_i'+\gamma_i^+],\; r\in[0,1]\big\}\cap X_{p/q+}=\{z_i\}.$
\item  Denote $$ \Gamma_i:=\{(s,\varphi_i(s)): s\in[s_i'+\gamma_i^-,s_i'+\gamma^+_i]\},\quad  i=\pm2,$$
and $$\Gamma_0^{\pm}:=\{(s,\varphi_0^{\pm}(s): s\in[s_0'+\gamma_0^-,s_0'+\gamma_0^+].$$
 The graphs $\Gamma_0^-$ and $\Gamma_{-2}$ are the local graphs of  the unstable manifold of $x_0$ near the points $z_{i}$, $i=0,-2$, 
and the graphs $\Gamma_0^+$ and $\Gamma_2$ are the local graphs of the stable manifold of $x_1$ near the points $z_i$, $i=0, 2$.
 \item There exist  strictly increasing $C^{\tau}$ functions $$\eta_i(t):[s_{i}'+\gamma_{i}^-,s_{i}'+\gamma_{i}^+]\to[s_{i+2}'+\gamma_{i+2}^-,s_{i+2}'+\gamma_{i+2}^+],\; i=-2,0$$
such that $ \eta(s_{i}')=s_{i+2}'$, $i=-2,0$,
$$f^2\big(s,\varphi_{-2}(s)\big)=\big(\eta_{-2}(s),\varphi_{0}^-(\eta_{-2}(s))\big),\quad s\in:[s_{-2}'+\gamma_{-2}^-,s_{-2}'+\gamma_{-2}^+],$$
and
$$ f^2\big(s,\varphi_0^+(s))=(\eta_0(s),\varphi_2(\eta_0(s))\big),\quad s\in:[s_{0}'+\gamma_{0}^-,s_{0}'+\gamma_{0}^+].$$
\end{enumerate}

Let  $0<\bar{\epsilon}_1< \frac{1}{3}\min\{|\gamma_{-2}^+|,|\gamma_{-2}^-|\}$, $0<\bar{\epsilon}_2<\frac{1}{3}\min\{|\gamma_0^+|,|\gamma_0^-|\}$ and $\bar{\epsilon}_1$, $\bar{\epsilon}_2$ be small enough. 
For $\epsilon_1\in[-\bar{\epsilon}_1,\bar{\epsilon}_1]$ and $\epsilon_2\in[-\bar{\epsilon}_2,\bar{\epsilon}_2]$, we  define a deformation  $\Omega_{\epsilon_1,\epsilon_2}$ of the domain $\Omega$, with the corresponding billiard map $f_{\epsilon_1,\epsilon_2}$ such that
\begin{enumerate}
\item[i.] If  $s\in[s_{-2}'-\epsilon_1,s_{-2}'+\epsilon_1]$ and $r=\varphi_{-2}(s)$, then $$f^2_{\epsilon_1,\epsilon_2}(s,r)=\big(\eta_{-2}(s+\epsilon_1),\varphi_0^-(\eta_{-2}(s+\epsilon_1))\big).$$
\item[ii.] If $s\in[s_{0}'-\epsilon_2,s_0'+\epsilon_2]$, and $r=\varphi_0^{+}(s)$,  then $$f^2_{\epsilon_1,\epsilon_2}(s,r)=\big(\eta_0(s+\epsilon_2),\varphi_2(\eta_1(s+\epsilon_2)\big).$$
\item[iii.]  Let $$\epsilon_1'=\max\{\big|\pi_1(f(s_{-2}\pm\epsilon_1,\varphi_{-2}(s_{-2}\pm\epsilon_1))-s_{-1}'\big|\}$$and$$\epsilon_2'=\max\{\big|\pi_1(f(s_0\pm\epsilon_2,\varphi_{0}^+(s_{0}\pm\epsilon_2))-s_1'\big|\}.$$If $ s\not\in [s_{-1}'-3\epsilon_1',s_{-1}'+3\epsilon_1']\cup[s_1'-3\epsilon_2',s_1'+3\epsilon_2']$, then
$$\partial\Omega_{\epsilon_1,\epsilon_2}(s)=\partial\Omega(s).$$
\end{enumerate}
The existence of such domain is due to the implicit function theorem for small enough $\bar{\epsilon}_1$ and $\bar{\epsilon}_2$.

By the construction, we could see that  for $f_{\epsilon_1,\epsilon_2}$:
  
\begin{enumerate}
\item[a.] $X_{p/q}$ is still the minimal periodic orbit in $\M_{\frac{p}{q}}$; 
 \item[b.] the orbit $\{f_{\epsilon_1,\epsilon_2}^i(z_0), i\in\mathbb{Z}\}$ is the minimal orbit in $\M_{\frac{p}{q}+}$;
 \item[c.] near $X_{p/q}$, the billiard maps $f_{\epsilon_1,\epsilon_2}$ and $f$ are the same;
 \item[d.] the point $f^{-n_0q+1}_{\epsilon_1,\epsilon_2}(z_0) $ moves non-degenerately as $\epsilon_1$ change. So does the point $f^{m_0q}_{\epsilon_1,\epsilon_2}$ with respect to $\epsilon_2$.
 \end{enumerate}
 These imply that the parametrized family of billiard maps $f_{\epsilon_1,\epsilon_2}$ satisfy the requirements of the Lemma.
 \end{proof}
 
 \vspace{10 pt}

\begin{lemma} \label{generic-lemma}For a generic billiard map $f$, we have that   for each $p/q\in\mathbb{Q}\cap(0,1/2)$, the constants $C_{p,q}(f)$ and $C_{p,q}'(f)$ in Theorem \ref{lemma-apr} are not zero.\\
\end{lemma}

\begin{proof}For each $p/q\in\mathbb{Q}\cap(0,1/2]$,  let us denote $\mathcal{G}_{p/q}$ the set of billiard maps $f$ such that $C_{p/q}(f)\neq0$ and $C_{p/q}'(f)\neq0$.  Clearly $\mathcal{G}_{p/q}$ is an open set, since $C_{p/q}(f)$ and $C_{p/q}'(f)$ are continuous with respect to $f$ in the $C^{\tau}$-topology.  If $C_{p/q}(f)=0$, by Lemma~\ref{lemma-parametrized}, we could find a billiard map $f'$, which is arbitrary close to $f$ in the $C^{\tau}$-topology, such that  $C_{p/q}(f')\neq0$. Therefore $\mathcal{G}_{p/q}$ is a dense open subset.
Then we can choose the generic set  to the residual set  
\begin{equation}\label{Gprime}
\mathcal{G}'=\cap_{p/q\in\mathbb{Q}\cap(0,1/2]}\mathcal{G}_{p/q}.
\end{equation}
In particular,
 each billiard map $f\in\mathcal{G}'$  verifies  the assertion  of the Lemma.
\end{proof}

\bigskip

{We can now conclude this section by proving assertion (2) in Main Theorem.}

{\begin{proof}[{\bf Main Theorem, item (2)}]
Let us denote $\mathcal{E}'$ the set of strictly convex billiard tables, for which the induced billiard maps  belong to $\mathcal{G}'$, {as defined in \eqref{Gprime}}. Consider the set $$\mathcal{E}=\mathcal{E}'\cap\mathcal{E}^{\tau}.$$
Clearly, $\mathcal{E}$ is a residual set. Then the assertion  (2) of Main Theorem follows from Lemmas~\ref{log-limit} and \ref{generic-lemma}. This concludes the proof.
\end{proof}
}

\vspace{10 pt}

\brm \label{rm:general-case2} In order to extend   Main Theorem from Aubry-Mather periodic orbits   
to arbitrary hyperbolic periodic orbits of a generic domain (namely, determine the eigenvalue of the linearization 
of the associated Poincare return map from the Marked Length Spectrum),  we  face two types of difficulties.  
\begin{itemize}
\item By a result in \cite{XZ}, for a hyperbolic periodic orbit there is a homoclinic orbit, which 
is generically transverse. Existence of a transverse homoclinic orbit implies existence of a sequence 
of hyperbolic periodic orbits accumulating to it. To proceed with our scheme, we need to determine the corresponding 
sequence in the Marked Length Spectrum. In the light of Remark \ref{rm:general-case1}, this should provide 
Theorem  \ref{lemma-apr}. 
\item  In order to prove Lemma \ref{log-limit}, we need to know that constant $C_{p,q}$ and $C'_{p,q}$ are 
non-zero. In Lemma \ref{lemma-parametrized} we essentially use the graph property of $p/q+$ orbits,
which is however not true in general.  \\
\end{itemize}
\erm


\end{document}